\documentclass[11pt,twoside]{article}
\usepackage{latexsym}
\usepackage{amssymb,amsbsy,amsmath,amsfonts,amssymb,amscd}
\usepackage{graphicx,xcolor}
\usepackage{float,url}
\usepackage{cancel}
\usepackage{mathtools}
\usepackage[colorlinks,linkcolor=red,citecolor=blue]{hyperref}
\usepackage{enumitem}
\usepackage{amsthm}
\usepackage{tikz}
\usepackage{dsfont}
\usetikzlibrary{arrows.meta,calc,positioning,shadows,shapes}
\usepackage[sort]{cite}

\tikzstyle{line} = [draw]

\tikzset{label/.style={draw=gray, ultra thin, rounded corners=.25ex, fill=gray!20,text width=4cm, text badly centered,  inner sep=2ex, anchor=east, minimum height=4em}}

\newcommand{\commentout}[1]{}

\newcommand{\weaks}{\overset{\ast}{\rightharpoonup}}
\newcommand{\R}{\mathbb{R}}

\newcommand{\ie}{\emph{i.e.}\;}
\newcommand {\al} {\alpha}
\newcommand {\eps}  {\varepsilon}

\newcommand {\dv}  { {\rm div} }
\newcommand {\cae} { {\mathcal E} }
\newcommand {\caf} { {\mathcal F} }

\newcommand {\f}   {\frac}
\newcommand {\p}   {\partial}
\newcommand{\beq}{\begin{equation}}
\newcommand{\eeq}{\end{equation}}
\newcommand{\diff}{\mathop{}\!\mathrm{d}} 
\DeclarePairedDelimiter{\norm}{\lVert}{\rVert}
\newtheorem{theorem}{Theorem}
\newtheorem{lemma}[theorem]{Lemma}

\newtheorem{remark}[theorem]{Remark}
\newtheorem{proposition}[theorem]{Proposition}

\title{From Vlasov equation to degenerate nonlocal Cahn-Hilliard equation}

\author{Charles Elbar\footnotemark[1]\thanks{Sorbonne Universit\'{e}, CNRS, Universit\'{e} de Paris, Inria, Laboratoire Jacques-Louis Lions (LJLL), F-75005 Paris, France } \thanks{email: charles.elbar@sorbonne-universite.fr}
\and Marco Mason\thanks{Sapienza Università di Roma, Dipartimento di Matematica "Guido Castelnuovo", Rome, Italy} \thanks{email: mason.1820840@studenti.uniroma1.it}
\and Beno\^it Perthame\footnotemark[1]  \thanks{email: benoit.perthame@sorbonne-universite.fr}
\and Jakub Skrzeczkowski\thanks{Institute of Mathematics Polish Academy of Sciences, Warsaw, Poland; Faculty of Mathematics, Informatics and Mechanics, University of Warsaw, Warsaw, Poland} \thanks{email: jakub.skrzeczkowski@student.uw.edu.pl}
}
\date{\today}

\begin{document}
\maketitle
\pagestyle{plain}
\pagenumbering{arabic}

\begin{abstract} 
We provide a rigorous mathematical framework to establish the hydrodynamic limit of the Vlasov model introduced in~\cite{takata2018simple} by Noguchi and Takata in order to describe phase transition of fluids by  kinetic equations. We prove that, when the scale parameter tends to 0, this model converges to a nonlocal Cahn-Hilliard equation with degenerate mobility. For our analysis, we introduce apropriate forms of the short and long range potentials which allow us  to derive Helmhotlz free energy estimates. Several compactness properties follow from the energy, the energy dissipation and kinetic averaging lemmas. In particular we prove a new weak compactness bound on the flux.  
\end{abstract} 
\vskip .7cm

\noindent{\makebox[1in]\hrulefill}\newline
\section{Introduction}
\label{sec:intro}

We consider the following Vlasov-Cahn-Hilliard equation (VCH in short) 
\beq \begin{cases}
\eps^2\p_t f_\eps +\eps  \xi.\nabla_x f_\eps + \eps F_\eps.\nabla_\xi f_\eps = \varrho_\eps (t,x)  M(\xi)  - f_\eps, \qquad t \geq 0, \; x \in \R^d, \; \xi  \in \R^d,
\\[4pt]
\varrho_\eps (t,x)= \int_{\R^d} f_\eps (t,x, \xi) \diff \xi,
\end{cases}
\label{eq:V}
\eeq
with an initial data $f_{\eps}(0,x,\xi)=f^0(x, \xi) \geq 0$.   \noindent The unknown is the 
function
\begin{equation*}
f_{\eps} \equiv f_{\eps}(t,x,\xi), \quad t \in (0,T), \ x\in \R^d, \ \xi \in \R^d,
\end{equation*}
such that, for every infinitesimal volume $\diff x \diff \xi$ around the point $(x,\xi)$ in the phase space, the quantity $f_{\eps}(t,x,\xi) \diff x \diff \xi$ is the number of particles which have position $x$ and velocity $\xi$ at fixed time $t$. The small parameter $\eps> 0$ arises from physical dimensions of the system and we are interested in the limit when it tends to 0.
Following~\cite{takata2018simple}, the force field $F_\eps(t,x)$ is decomposed as long-range attractive and short-range repulsive
\begin{equation} \label{eq:potential}
F_\eps = F_\eps^L+F_\eps^S, \qquad F_\eps^{L,S}(t,x) = -\nabla \Phi_\eps^{L,S}(t,x).
\end{equation}
We define the convolution in the space variable as $f\star g =  \int_{\R^d} f(y) g(x-y) \diff y$ and set
\[
\Phi_{\alpha, \eps}^S(t,x) =\frac{1}{\al^{2}} \omega^S \star  \omega^S \star \varrho_\eps,
\]
where $\omega^S\geq 0$ is a function that may be thought of as a centered Gaussian. We use a double convolution in order to enforce positivity of the corresponding operator as it appears in energy considerations. We assume that $\omega^S$ satisfies  
\begin{equation} \label{as:omega_S}
\int_{\R^d}\omega^S(y) \diff y =1 , \quad \int_{\R^d}y \omega^S(y) \diff y = 0, \quad \int_{\R^d}|y|^2 \omega^S(y) \diff y < \infty.
\end{equation}
The long-range potential is of the form 
\begin{equation} \label{as:alpha}
\Phi_{\alpha, \eps}^L(t,x) = -\frac{1}{\al^{2}} \omega^L_{\al}\star \omega^S \star \omega^S \star \varrho_\eps, 
\end{equation}
where $\omega^L_\alpha(x) = \frac{1}{\alpha^d} \, \omega^L\left(\frac{x}{\alpha}\right)$ may be thought of as a high temperature Gaussian and $\omega^L$ is a smooth, nonnegative, symmetric, compactly supported function such that, for some $\delta>0$,
\begin{equation}
 \int_{\R^d}\omega^L(y) \diff y =1, \quad \int_{\R^d} y \omega^L(y) \diff y =0,  \quad \int_{\R^d}   y_i y_j \omega^L \diff y = \delta_{i,j}\, \delta, \quad \int_{\R^d} \omega^L(y) |y|^3 \diff y < \infty.
 \label{as:omega}
\end{equation}

\noindent The equilibrium distribution $M(\xi) \geq 0$ is a Maxwellian that we normalize as
\begin{equation} \label{eq: maxwellian}
 M(\xi) := \left(\frac{1}{2\pi D}\right)^{d/2} \exp\left(-\f{|\xi|^2}{2D}\right),
\end{equation}
and we have, for $i=1,\dots, d$,
\begin{equation} \label{eq: maxwellian_properties}
\int_{\R^d} M(\xi)\diff \xi=1, \qquad \int_{\R^d} \xi_i M(\xi)\diff \xi=0, \qquad
\int_{\R^d} \xi_i^2 M(\xi) \diff \xi=D<\infty, 
\end{equation}
\medskip
so that $D$ can be interpreted as the diffusion coefficient. 

\subsection{The macroscopic limit}
%
 The right-hand side of Equation~\eqref{eq:V} is a relaxation term that conserves mass but neither momentum nor energy since we aim at using a diffusive scaling. Formally one can guess that
\begin{equation}
f_{\eps}(t,x,\xi)\to \varrho(t,x)M(\xi),\quad \text{as $\eps\to 0.$}    
\label{eq:asympf}
\end{equation}
The mass conservation equation on $\varrho_{\eps}$  is obtained by integrating Equation~\eqref{eq:V} with respect to~$\xi$ against~$1$,
\beq
\p_t \varrho_\eps (t,x) + \dv J_\eps (t,x) = 0, \qquad  J_\eps (t,x) =   \int_{\R^d} \f \xi \eps f_\eps (t,x, \xi) \diff \xi .
\label{eq:Mass}
\eeq
Then,  integrating  against $\xi$, we obtain the flux equation 
\beq
\eps^2 \p_t  J_\eps (t,x) + \nabla_{x}\cdot  \int_{\R^d} \xi \otimes \xi   f_\eps (t,x, \xi) \diff \xi -  F_\eps \varrho_\eps = - J_\eps (t,x).
\label{eq:Jeps}
\eeq
Combined with \eqref{eq:asympf}, this flux equation allows us to identify the limit of $J_\eps $ and to prove that as $\eps,\alpha\to 0$, the macroscopic densities tend to a solution of a degenerate nonlocal Cahn-Hilliard equation type. More precisely, we have the

\begin{theorem}[Limit $\varepsilon \to 0$] \label{thm:limit_eps}
With the assumptions and notations \eqref{eq:potential}--\eqref{eq: maxwellian}, let $\alpha=\eps$. Let $f^{0}$ be a non-negative distribution that satisfies~\eqref{initial_data}-\eqref{initial_helmholtz} and let $f_{\eps}$ be a solution of~\eqref{eq:V} with initial condition $f^{0}$. Then, we can extract a subsequence (not relabelled) such that $\varrho_\eps \to \varrho$ in $L^p_{t}L^{1}_{x}$ strongly for $1\le p<\infty$ where $\varrho$ solves in the distributional sense the equation
\begin{equation}\label{limit_equation}
\p_t \varrho -D\Delta \varrho - \dv (\varrho \nabla \Phi)=0, \qquad \Phi = -\delta\Delta[\omega^{S}\star\omega^{S}\star\varrho],
\end{equation}
with initial data $\varrho^0= \int_{\R^d} f^0 (x, \xi) \diff \xi$.
\end{theorem}

 In fact,~\cite{takata2018simple} obtains formally a more complete description which we cannot prove at the moment (see Section~\ref{conclusion}).

\begin{remark}
$\bullet$  Writing formally $\Delta\varrho=\dv(\varrho\nabla\log(\varrho))$,  this term can be added to the potential so as to obtain a kind of Cahn-Hilliard equation. 
\\ $\bullet$ Different scaling between $\alpha$ and $\eps$ can be considered, $\alpha$ constant is also possible
\\ $\bullet$  Uniqueness can be proved in the class of uniformly bounded densities, see Appendix~\ref{app:uniqueness}. 
\end{remark}

\subsection{Contents of the paper}

In Section~\ref{sec:EE}, we collect various uniform estimates $\eps$.  Section~\ref{sec:limit_2} is devoted to passing to the limit $\eps\to 0$. Some open problems are drawn in Section~\ref{conclusion}. The Appendix contains different mathematical tools and lemmas used throughout the proofs.

\subsection{Literature review and relevancy of the system}

\paragraph{Phase transitions in fluids}
 In~\cite{takata2018simple}, Noguchi and Takata consider a kinetic model to capture the dynamics of phase transition for the Van der Waals fluid. The model reads as follows

\begin{align*}
&\p_{t}f+\xi\cdot\nabla_{x}f+F\cdot\nabla_{\xi}f=A(\varrho)(\varrho M-f), \quad A(\varrho)>0,\\
&F=F^{1}+F^{2},\quad F^{1,2}=-\nabla\Phi^{1,2},\\
&\Phi^{1}=\begin{cases}\varrho-\omega^{L}\star\varrho\\ \text{or }-\kappa\Delta\varrho\end{cases},\quad \Phi^{2}=-C_{1}\log(1-\varrho)+\f{\varrho}{1-\varrho}-C_{2}\varrho,\quad C_{1}, C_{2}\in \R, 
\end{align*}
for some kernel $\omega^{L}$ and $\kappa>0$. $\Phi^{1}$ is a combination of short range repulsion and long range attraction. $\Phi^{2}$ is a short range interaction potential. 

The authors state that full details of intermolecular collisions are not considered and that the collision term on the right-hand side plays just a thermal bath role. However, they show that the system exhibits the essential features of phase transition dynamics, both theoretically and
numerically. By placing themselves in the framework of the strong interaction, they find a rescaling of the first equation of the system and obtain Equation~\eqref{eq:V} of VCH. Then, setting $A\equiv 1$ and letting $\eps\to 0$ they obtain formally that in the limit (that we refer to as the hydrodynamic limit), the macroscopic density $\varrho$ satisfy

\begin{equation*}
\p_{t}\varrho-\Delta\varrho-\dv(\varrho\nabla(\Phi^{1}+\Phi^{2}))=0.    
\end{equation*}

Noting that $\Delta\varrho=\dv(\varrho\nabla\log(\varrho))$ they obtain the Cahn-Hilliard equation with degenerate mobility 

\begin{equation*}
\p_{t}\varrho-\dv(\varrho\nabla(\Phi^{1}+\widetilde{\Phi}^{2}))=0,\quad \widetilde{\Phi}^{2}=\Phi^{2}+\log(\varrho).    
\end{equation*}

This model presents several mathematical difficulties. First of all, we are not aware of any existence result concerning the Vlasov equation when the potential $\Phi$ is a function of the density $\varrho$. In the Vlasov-Poisson system, one has $\Phi=\Delta^{-1}\varrho$ and there is a  gain of two derivatives. For the existence of classical solutions for Vlasov-Poisson we refer to~\cite{MR1115549,MR1165424,MR1132787,MR1200156}. A second difficulty comes from the rigorous passage to the limit. Indeed, the bound provided by the energy do not provide enough compactness. For instance, one cannot apply the averaging lemma~\ref{lem:masmoudi} on this system because the functions are not bounded in $L^{1}$ uniformly in $\eps$.   
For these reasons, we add the convolutions $\omega^{S}$ in~\eqref{eq:V}--\eqref{as:alpha} and provide a rigorous mathematical framework to establish the hydrodynamic limit of this model when $\Phi^{2}=0$. It would be possible to prove a similar result when $\Phi^{2}=\omega^{S}\star f'(\varrho)$ where $|f(\varrho)|\le C|\varrho\log\varrho|$ for $C$ small enough.  

Our work also provides a generic model to obtain different nonlocal and degenerate equations of Cahn-Hilliard/thin-film type as the hydrodynamic limit of kinetic models. 
\noindent For other kinetic models modeling phase transitions, we refer to \cite{giovangigli,Grmela1974OnTA,kobayashi2012}.

\paragraph{Kinetic theory}

\vspace{0.3cm} \noindent The main purpose of kinetic theory is to provide a description of the evolution of a gas or plasma, and more generally a many-particle system made up of $N$ similar individual elements, in the limit when $N$ tends to infinity which corresponds to the so-called thermodynamical limit. 

\vspace{0.3cm} \noindent In the kinetic theory, the density of particles is described with the probability measure
\begin{equation*}
f \equiv f(t,x,\xi), \quad t\ge 0, \ x\in \R^d, \ \xi \in \R^d,
\end{equation*}
such that, for every infinitesimal volume $\diff x \diff \xi$ around the point $(x,\xi)$ in the phase space, the quantity $f(t,x,\xi) \diff x \diff \xi$ is the number of particles which have position $x$ and velocity $\xi$ at fixed time $t$. For this reason, $f$ is a nonnegative function and integrable in both space and velocity variables, but it is not directly observable. Nevertheless, at each point of the domain it provides all measurable macroscopic quantities which can be expressed in terms of microscopic averages:
\begin{align*}
\varrho(t,x)&=\int_{\R^d} f(t,x,\xi) \diff \xi \quad \quad \mbox{(macroscopic density)}, \\
J(t,x)&= \int_{\R^d} \xi f(t,x,\xi) \diff \xi \quad \quad \mbox{(flux)}.
\end{align*}
It is clear that such a statistical description makes sense only with a very large number of particles, and as a consequence, all kinetic equations are expected to approximate the true dynamics of gases just in the thermodynamical limit. Rescaling the time and space with a parameter $\eps$, \ie $t\to\eps^{2}t$, $x\to\eps x$ and sending $\eps\to 0$ is called the hydrodynamic limit. It allows us to find a rigorous derivation of macroscopic models from a microscopic description of matter. For hydrodynamics on the Vlasov-Poisson-Fokker-Planck system, we refer to~\cite{MR2664460,MR2139941}.\\ 

\noindent Our aim is to obtain an equation on the macroscopic density and to relate it to a known model that has applications in fluid dynamics or biology, \ie the Cahn-Hilliard equation.

\paragraph{The Cahn-Hilliard equation}

Equation~\eqref{limit_equation} is an example of a Cahn-Hilliard type equation that is widely used nowadays to represent phase transitions in fluids and living tissues~\cite{wise_three-dimensional_2008,Frieboes-2010-CH,Garcke-2016-CH-darcy,Garcke-2018-multi-CH-darcy,Frigeri-CH-Darcy,ebenbeck_cahn-hilliard-brinkman_2018,Ebenbeck-Brinkman,perthame_poulain,elbar-perthame-poulain,Lowengrub-bridging}. Originally introduced in the context of materials sciences \cite{Cahn-Hilliard-1958, CAHN1961795}, it is currently applied in numerous fields, including complex fluids, polymer science, and mathematical biology.  For the overview of mathematical theory, we refer to \cite{MR4001523}.\\

Cahn-Hilliard equation takes the form of 
\begin{equation}
    \p_t \varrho = \dv\left(b(\varrho)\nabla \left(\Phi(\varrho)-\delta \Delta \varrho \right)\right) \to \begin{cases}
        \p_t \varrho &= \dv\left(b(\varrho)\nabla \mu \right),\\
        \Phi &= -\delta \Delta \varrho +\Phi^{I}(\varrho),
    \end{cases}
    \label{eq:Cahn-Hilliard}
\end{equation}
where $\varrho$ represents the relative density of one component $\varrho=\varrho_1/(\varrho_1+\varrho_2)$, $b(\varrho)$ is the mobility, $\Phi^{I}$ is the interaction potential while $\Phi$ is the quantity of chemical potential.\\

We obtain a nonlocal version of the Cahn-Hilliard equation. The nonlocality comes from the convolution of the Laplace operator with a smooth kernel $\omega^{S}$ concentrated around the origin. There is a different possibility to approximate this operator nonlocally, we refer for instance to~\cite{MR4198717,MR4408204}, where the authors prove the convergence of a nonlocal Cahn-Hilliard equation with constant mobility to a local Cahn-Hilliard equation. In our case, because of the degenerate mobility, it is not clear that we can pass from the nonlocal Equation~\eqref{limit_equation} to a local one by sending $\omega^{S}$ to a Dirac mass.

\section{Entropy, energy, and uniform estimates}
\label{sec:EE}
The analysis relies on various uniform bounds in $\eps$ which use an initial data that satisfies
\begin{equation} \label{initial_data}
\int_{\R^{2d}} (1+|x|+|\xi|^2 +|\log f^0|)  f^0(x,\xi) \diff x \diff \xi < +\infty, 
\end{equation}
\begin{equation} \label{initial_helmholtz}
\sup_{\al \leq 1} \; \frac {1}{\alpha^{2}}  \int_{\R^{2d}} \omega^L_\alpha (y) [\varrho^0 \star \omega^S (x)-\varrho^0 \star \omega^S (x- y)]^2 \diff x \diff y < +\infty.
\end{equation}

Then, we begin with proving the bounds 
\begin{theorem}[Uniform estimates]\label{thm:bounds_epsilonto0}
With the assumptions \eqref{initial_data} and \eqref{initial_helmholtz}, the following uniform estimates hold for $\eps \in (0,1)$:
\begin{enumerate}[label=(\Alph*)]
    \item  \label{thm_eps:est1} $\{f_\varepsilon\}$ in $L^{\infty}_{t} L^1_{x,\xi}$ and $\{\varrho_{\eps}\}$ in $L^{\infty}_t L^1_x$,
    \item \label{thm_eps:est2} $\{f_{\varepsilon} |\log(f_{\varepsilon})|\}$ and $\{f_{\varepsilon} \, |\xi|^2\}$ in $L^{\infty}_{t} L^1_{x,\xi}$,
    \item \label{thm_eps:est3} $\{\varrho_{\varepsilon} |\log(\varrho_{\varepsilon})|\}$ in $L^{\infty}_{t} L^1_{x}$,
    \item \label{thm_eps:est4} $\left\{
    \frac{(\varrho_{\varepsilon} M - f_{\varepsilon}) \, (\log(\varrho_{\varepsilon} M) - \log(f_{\varepsilon}))}{\varepsilon^2}\right\}$ in $L^1_{t,x,\xi}$,
    \item \label{thm_eps:est5} $\left\{
    \frac{\varrho_{\varepsilon} M - f_{\varepsilon}}{\varepsilon}\right\}$ in $L^1_{t,x,\xi}$,
    \item \label{thm_eps:est6} $\{ J_{\varepsilon}\}$ and $\{ J_{\varepsilon}\log^{1/2}\log^{1/2} \max(J_{\varepsilon},e) \}$ in $L^1_{t,x}$,
    \item \label{thm_eps:est7} $\{f_{\eps} |x|\}$, $\{\varrho_{\eps}\, |x|\}$ in $L^{\infty}_t L^1_{x,\xi}$ and $L^{\infty}_t L^1_{x}$ respectively.
\end{enumerate}
Moreover, $\{\varrho_{\eps}\}$ and $\{ J_{\varepsilon}\}$ are weakly compact in $L^1_{t,x}$.
\end{theorem}

\noindent The proof of these estimates uses a fundamental property of energy dissipation. To show that, we define the energy (kinetic+potential) and the Helmholtz free energy respectively as
\begin{align} \label{def:energy}
\cae (t) :=\int_{\R^{2d}} |\xi|^2 f_\eps \diff x \diff \xi  +\frac {1}{2\alpha^{2}} \int_{\R^{2d}} \omega^L_\alpha (y) [\varrho_\eps \star \omega^S (t,x)-\varrho_\eps \star \omega^S (t,x- y)]^2 \diff x \diff y,
\end{align}
\begin{align} \label{def:Henergy}
\caf (t) :=\int_{\R^{2d}} [2D\log(f_\eps)+|\xi|^2] f_\eps \diff x \diff \xi  +\frac {1}{2\alpha^{2}}  \int_{\R^{2d}} \omega^L_\alpha (y) [\varrho_\eps \star \omega^S (t,x)-\varrho_\eps \star \omega^S (t,x- y)]^2 \diff x \diff y.
\end{align}
The Helmholtz free energy satisfies the

\begin{theorem} [Free energy dissipation] \label{thm: helmholtz} 
The free energy $\ \caf (t)$ is dissipated as
\begin{equation} \label{helmholtz_dissipation}
\f {d}{dt} \caf (t)= -\f{2D}{\eps^2} \int_{\R^{2d}} \left[f_\eps-\varrho_\eps M(\xi)\right]\left[\log f_\eps-\log \left(\varrho_\eps M(\xi)\right)\right] \diff x \diff \xi =-2D\int_{\R^{2d}} {\mathcal D}_\eps \diff x \diff \xi,
\end{equation}
where the dissipation term is defined as 
\begin{equation} \label{eq:dissipation_term}
{\mathcal D}_\eps(t,x,\xi):=\f{1}{\eps^2} \left[f_\eps-\varrho_\eps M(\xi)\right]\left[\log f_\eps-\log \left(\varrho_\eps M(\xi)\right)\right] \geq 0.
\end{equation}
\end{theorem}

\noindent This theorem  can be seen as a combination of relations for both the total energy and the entropy of the system.

\begin{proposition} [Total energy dissipation] \label{prop:total_energy_dissipation}
The total energy $\ \cae (t)$ is dissipated as
\begin{equation} \label{energy_dissipation}
\f {d}{dt} \cae (t)= \f{1}{\eps^2}  \int_{\R^{2d}} |\xi|^2 \left[ \varrho_\eps M(\xi)- f_\eps \right] \diff x \diff \xi.
\end{equation}
\end{proposition}
\begin{proof}
By multiplying \eqref{eq:V} by $|\xi|^2$ and taking the integrals with respect to $x$ and $\xi$ we obtain
\begin{equation} \label{xi^2-moment}
\begin{split} 
\eps^2 \int_{\R^{2d}} |\xi|^2 \p_t f_\eps \diff x \diff \xi + \eps \int_{\R^{2d}} |\xi|^2 \xi \cdot \nabla_x f_\eps \diff x \diff \xi & +\eps \int_{\R^{2d}} |\xi|^2 F_\eps \nabla_\xi f_\eps \diff x \diff \xi \\
&=  \int_{\R^{2d}} |\xi|^2 [\varrho_\eps M(\xi)-f_\eps] \diff x \diff \xi.
\end{split}
\end{equation}
For integrable solutions, the second term on the left-hand side vanishes. Furthermore, with integration by parts, the above equation reduces to
\begin{equation} \label{xi^2-moment after integration by part}
\eps^2 \f{d}{dt} \int_{\R^{2d}} |\xi|^2 f_\eps \diff x \diff \xi -2\eps \int_{\R^{2d}} \xi F_\eps f_\eps \diff x \diff \xi = \int_{\R^{2d}} |\xi|^2 [\varrho_\eps M(\xi)-f_\eps] \diff x \diff \xi.
\end{equation}
By recalling \eqref{eq:Mass}, the second term can be rewritten as
\begin{equation}
-2\eps \int_{\R^{2d}} \xi F_\eps f_\eps \diff x \diff \xi=
-2\eps^2 \int_{\R^d} \Phi_{\alpha, \eps} \,\dv J_\eps  \diff x= 2\eps^2 \int_{\R^d} \Phi_{\alpha, \eps} \,\p_t \varrho_\eps \diff x.
\end{equation}
We now want to prove that 
\begin{equation} \label{second_term_derivative}
2 \int_{\R^d} \Phi_{\alpha, \eps} \p_t \varrho_\eps \diff x= \f{1}{2\alpha^2} \f{d}{dt} \int_{\R^{2d}} \omega_\alpha^L (y) [\varrho_\eps \star \omega^S (t,x)-\varrho_\eps \star \omega^S (t,x-y)]^2 \diff x \diff y.
\end{equation}
First, by recalling \eqref{eq:potential},
\begin{equation}
2\int_{\R^d} \Phi_{\alpha, \eps} (t,x)\p_t \varrho_\varepsilon (t,x) \diff x=
 2\int_{\R^d} \Phi_{\alpha, \eps}^L(t,x) \p_t \varrho_\varepsilon(t,x) \diff x
 +2\int_{\R^d} \Phi_{\alpha, \eps}^S (t,x)\p_t \varrho_\varepsilon(t,x) \diff x.
\end{equation}
As regards the first term on the right-hand side
\begin{equation*}
\begin{split}
2\int_{\R^d} \Phi_{\alpha, \eps}^L(t,x) \p_t \varrho_\varepsilon(t,x) \diff x &=-\f{2}{\alpha^2}\int_{\R^d} [\omega^L_\alpha\star\omega^S\star\omega^S\star \varrho_\varepsilon] (t,x) \p_t \varrho_\varepsilon (t,x)\diff x \\
&=-\f{2}{\alpha^2}\int_{\R^d} [\omega^L_\alpha\star \varrho_\varepsilon\star \omega^S] (t,x) \p_t [\varrho_\varepsilon \star \omega^S] (t,x)\diff x \\
& =-\f{1}{\alpha^2}\int_{\R^{2d}} \omega^L_\alpha(y) [\varrho_\varepsilon \star \omega^S] (t,x-y) \p_t [\varrho_\varepsilon \star \omega^S] (t,x)\diff x \diff y  \\
& \quad \ -\f{1}{\alpha^2}\int_{\R^{2d}} \omega^L_\alpha(y) [\varrho_\varepsilon \star \omega^S] (t,x) \p_t [\varrho_\varepsilon \star \omega^S] (t,x-y)\diff x \diff y \\
&=-\f{1}{\alpha^2}\frac{d}{dt} \int_{\R^{2d}} \omega^L_\alpha(y)[[\varrho_\varepsilon \star \omega^S](t,x)\cdot [\varrho_\varepsilon\star \omega^S](t,x-y)]\diff x \diff y
\end{split}
\end{equation*}
The second term on the right-hand side can be handled similarly and gives
\begin{eqnarray*}
2\int_{\R^d} \Phi_{\alpha, \eps}^S (t,x) \p_t \varrho_\varepsilon (t,x)\diff x &=
&\frac{2}{\alpha^2}\int_{\R^d} [\omega^S\star\varrho_\varepsilon](t,x) \p_t [ \varrho_\varepsilon \star \omega^S](t,x)\diff x \\&=&\frac{1}{\alpha^2}\frac{d}{dt}\int_{\R^d} [ \varrho_\varepsilon \star \omega^S]^2(t,x)\diff x \\&=&\frac{1}{2\alpha^2}\frac{d}{dt}\int_{\R^{2d}} \omega^L_\alpha(y) \big[ [ \varrho_\varepsilon \star \omega^S]^2(t,x) + [ \varrho_\varepsilon \star \omega^S]^2(t,x-y)\big]\diff x \diff y.
\end{eqnarray*}
By summing up the two previous identities we get~\eqref{second_term_derivative}, which, inserted in~\eqref{xi^2-moment after integration by part}, concludes that
\begin{equation*}
\begin{split}
\eps^2 \f{d}{dt} \int_{\R^{2d}} |\xi|^2 f_\eps \diff x \diff \xi +\f{\eps^2}{2\alpha^2} \f{d}{dt} \int_{\R^{2d}} \omega_\alpha^L (y) [\varrho_\eps \star \omega^S (t,x)-\varrho_\eps \star \omega^S (t,x-y)]^2 \diff x \diff y & \\
= \int_{\R^{2d}} |\xi|^2 [\varrho_\eps M(\xi)-f_\eps] \diff x \diff \xi.
\end{split}
\end{equation*}
\end{proof}

\begin{proposition} [Entropy relation] \label{prop:entropy_dissipation}
The following estimate holds:
\begin{equation} \label{entropy_dissipation}
\frac d {dt}   \int_{\R^{2d}}  f_\eps \log f_\eps \diff x \diff \xi = \f{1}{\eps^2} \int_{\R^{2d}}  [ \varrho_\eps M(\xi)-f_\eps]   \log f_\eps \diff x \diff \xi.
\end{equation}
\end{proposition}
\begin{proof}
By multiplying \eqref{eq:V} by $(1+\log f_\eps)$ we obtain
\begin{equation*}
\eps^2 \frac d {dt} (f_\eps \log f_\eps) + \eps \xi \cdot \nabla_x f_\eps (1+\log f_\eps) +\eps F_\eps \nabla_\xi f_\eps (1+\log f_\eps)=[\varrho_\eps M(\xi)-f_\eps](1+\log f_\eps)
\end{equation*}
By taking the integrals with respect to $x$ and $\xi$, the second and third terms in the above equation vanish and we obtain
\begin{equation*}
\eps^2 \frac d {dt} \int_{\R^{2d}} f_\eps \log f_\eps \diff x \diff \xi = \int_{\R^{2d}} [\varrho_\eps M(\xi)-f_\eps]\log f_\eps \diff x \diff \xi
\end{equation*}
as announced.
\end{proof}

\noindent With these two estimates, we can finally prove Theorem~\ref{thm: helmholtz}.

\begin{proof} [Proof of Theorem \ref{thm: helmholtz}]
From Propositions \ref{prop:total_energy_dissipation} and \ref{prop:entropy_dissipation}, we get the following result:
\begin{equation} \label{eq: helmholtz_1}
\begin{split}
\f{d}{dt} \caf(t)&=\f{1}{\eps^2}\left[ \int_{\R^{2d}} |\xi|^2 [\varrho_\eps M(\xi)-f_\eps] \diff x \diff \xi + 2D\int_{\R^{2d}} [ \varrho_\eps M(\xi)-f_\eps]   \log f_\eps \diff x \diff \xi \right] \\
& =\f{1}{\eps^2}\left[\int_{\R^{2d}} \varrho_\eps M(\xi) |\xi|^2 \diff x \diff \xi -\int_{\R^{2d}}|\xi|^2 f_\eps \diff x \diff \xi + 2D\int_{\R^{2d}} [ \varrho_\eps M(\xi)-f_\eps]   \log f_\eps \diff x \diff \xi \right].
\end{split}
\end{equation}
Using \eqref{eq: maxwellian}, we know that
$
\log (\varrho_\eps M(\xi)) =\log \varrho_\eps + C - \f{|\xi|^2}{2D}
$
for some constant $C$. Inserting this expression of $|\xi|^2$  in the first two terms on the righthand side of~\eqref{eq: helmholtz_1}, we obtain
\begin{equation*}
\begin{split}
\f{2D}{\eps^2}\int_{\R^{2d}}\left[\varrho_\eps M(\xi)-f_\eps\right]&\left[ \log \varrho_\eps + C -\log (\varrho_\eps M(\xi))\right] \diff x \diff \xi \\
&=\f{2D}{\eps^2}\int_{\R^{2d}} \left[\varrho_\eps M(\xi)-f_\eps\right]\left[ - \log (\varrho_\eps M(\xi))\right]  \diff x \diff \xi.
\end{split}
\end{equation*} 
Added to the third term on the righthand side of~\eqref{eq: helmholtz_1}, we obtain the announced result.
\end{proof}

\noindent In order to prove Theorem~\ref{thm:bounds_epsilonto0}, a major difficulty is to estimate the flux $J_\eps$ defined by \eqref{eq:Mass}. We start by establishing a useful inequality, recalling the notation \eqref{eq:dissipation_term}.

\begin{lemma}[Pointwise estimates on $J_{\eps}$] \label{lem:pointwise_J} For every $0<r\le 1$ and $(s,x)\in (0,T)\times\R^{d}$,  we have
\begin{equation*}
|J_{\eps}(s,x)|\le r \eps  \|{\mathcal D}_\eps(s, x,\cdot)\|_{L^1_{\xi}}+C\f{1}{r^{d}}\exp\left(\f{2C_{M}}{r^{2}}\right)\varrho_{\eps}(s,x)^{1/2}\|{\mathcal D}_\eps(s, x,\cdot)\|_{L^1_{\xi}}^{1/2}.
\end{equation*}
\end{lemma}

\begin{proof}
For $r>0$, we decompose $J_\eps(s,x) = J_{\eps}^{(1)}(s,x) + J_{\eps}^{(2)}(s,x)$, with
  
$$
J_\eps^{(1)} = \f{1}{\eps}\int_{\left\{\left|\log(\f{f_{\eps}}{\varrho_{\eps} M})\right|\ge\f{|\xi|}{r}\right\}}\xi (f_{\eps}-\varrho_{\eps}M(\xi))\diff \xi, \qquad
J_{\eps}^{(2)} = \f{1}{\eps}\int_{\left\{\left|\log(\f{f_{\eps}}{\varrho_{\eps} M})\right|\le\f{|\xi|}{r}\right\}}\xi (f_{\eps}-\varrho_{\eps}M(\xi))\diff \xi .
$$

For $J_\eps^{(1)}$, we write
\[
 |J_\eps^{(1)}(s,x)|  \le
\f{r}{\eps} \int_{\left\{\left|\log(\f{f_{\eps}}{\varrho_{\eps} M})\right|\ge\f{|\xi|}{r}\right\}}
\left|\log\left(\f{f_{\eps}}{\varrho_{\eps} M}\right)\right|\varrho_{\eps} M\left|\f{f_{\eps}}{\varrho_{\eps} M}-1\right|\diff\xi 
\le r \eps  \|{\mathcal D}_\eps(s, x,\cdot)\|_{L^1_{\xi}}.
\]

For $J_\eps^{(2)}$, we use the Cauchy-Schwarz inequality and, with $B(\xi):=\f{|\xi|}{r(\exp(\f{|\xi|}{r})-1)}$, 
\[
 |J_{\eps}^{(2)}(s, x)| \le \left(\int_{\R^d}|\xi|^{2}\f{\varrho_{\eps} M}{B(\xi)}\diff\xi \right)^{1/2}  \left(\f{1}{\eps^{2}}\int_{\left\{|\log(\f{f_{\eps}}{\varrho_{\eps} M})|\le\f{|\xi|}{r}\right\}}\varrho_{\eps} M\left|\f{f_{\eps}}{\varrho_{\eps} M}-1\right|^{2}B(\xi)\diff\xi \right)^{1/2}.
\]
Because $M(\xi)$ is a Gaussian and $\varrho_\eps$ depends only on $(t,x)$, we obtain
\[
 |J_{\eps}^{(2)}(s, x)| \le  \varrho_{\eps}(s, x)^{1/2} \left(\int_{\R^d}|\xi|^{2}\f{M(\xi)}{B(\xi)}\diff\xi \right)^{1/2}  (I_1+I_2)^{1/2}.
\]

Here we have split the second  integral according to the sign of $\log(\f{f_{\eps}}{\varrho_{\eps} M})$. When it is negative, we may write, since $B(\xi) \leq 1 $,
\[
I_1 := 
\f{1}{\eps^{2}}\int_{\left\{ \f{f_{\eps}}{\varrho_{\eps} M}\leq 1 \right\}}\varrho_{\eps} M\left|\f{f_{\eps}}{\varrho_{\eps} M}-1\right|^{2}B(\xi)\diff\xi
\leq \f{1}{\eps^{2}}\int_{\R^d}\varrho_{\eps} M\left|\f{f_{\eps}}{\varrho_{\eps} M}-1\right| \left|  \log \f{f_{\eps}}{\varrho_{\eps} M} \right|\diff\xi
= \|{\mathcal D}_\eps(s,x,\cdot)\|_{L^1_{\xi}}.
\]
The second term is defined as
\[
I_2 := \f{1}{\eps^{2}}\int_{\left\{ 0\leq \log(\f{f_{\eps}}{\varrho_{\eps} M})\leq \f{|\xi|}{r} \right\}}  \varrho_{\eps} M\left|\f{f_{\eps}}{\varrho_{\eps} M}-1\right|^{2}B(\xi)\diff\xi.
\]

Since $\log$ is a concave function, for $A>1$ and $y\in[1,A]$, we have $y-1 \leq \log(y)\f{A-1}{\log(A)}$. We choose $A=A(\xi):=\exp(\f{|\xi|}{r})$ and $y=\f{f_{\eps}}{\varrho_{\eps}M}$ so that $y \in [1,A]$ means exactly $0\leq\log(\f{f_{\eps}}{\varrho_{\eps} M})\le\f{|\xi|}{r}$. Then, $I_2$ can be estimated as follows

\begin{equation*}
I_{2}\le \f{1}{\eps^{2}}\int_{\R^d}\varrho_{\eps} M\left|\f{f_{\eps}}{\varrho_{\eps} M}-1\right|\log\left(\f{f_{\eps}}{\varrho_{\eps}M}\right)\f{r(\exp(\f{|\xi|}{r})-1)}{|\xi|}B(\xi)\diff\xi = \|{\mathcal D}_\eps(s,x,\cdot)\|_{L^1_{\xi}}.
\end{equation*}
Therefore, for some constant $C_{M}$, defined through $M(\xi)$, we have
\[
 |J_{\eps}^{(2)}(s, x)| \le  C\varrho_{\eps}(s, x)^{1/2} \|{\mathcal D}_\eps(s,x,\cdot)\|_{L^1_{\xi}}^{1/2}\left(\int_{\R^d}r|\xi|\exp\left(\f{-|\xi|^{2}}{C_{M}}\right)\left(\exp\left(\f{|\xi|}{r}\right)-1\right)\diff\xi \right)^{1/2}.
\]
It remains to treat the integral factor that we denote by~$I_{3}$ and for $r$ smaller than $1$,  
\begin{align*}
I_{3}&=\int_{\R^d}r|\xi|\exp\left(\f{-|\xi|^{2}}{C_{M}}\right)\left(\exp\left(\f{|\xi|}{r}\right)-1\right)\diff\xi
\le \f{C}{r^{d}}\exp\left(\f{2C_{M}}{r^{2}}\right)
\end{align*}
where $C$ does not depend on $r$. This can be seen by splitting the integral in the zones $\{|\xi|\le\f{2C_{M}}{r}\}$ and $\{|\xi|\ge \f{2C_{M}}{r}\}$. Finally, we obtain 
\begin{equation*}
|J_{\eps}|\le r \eps  \|{\mathcal D}_\eps(s, x,\cdot)\|_{L^1_{\xi}}+C\f{1}{r^{d}}\exp\left(\f{2C_{M}}{r^{2}}\right)\varrho_{\eps}(s,x)^{1/2}\|{\mathcal D}_\eps(s, x,\cdot)\|_{L^1_{\xi}}^{1/2}.
\end{equation*}
\end{proof}

From this lemma, we deduce the following $L^{1}$ bounds on $J_{\eps}$



\begin{proposition}[Estimate on $J_{\varepsilon}$ in $L^1_{x}$]\label{lem:estimate_Jeps}
With the decomposition of Lemma~\ref{lem:pointwise_J},  $
J_\eps(s,x) = J_{\eps}^{(1)}(s,x) + J_{\eps}^{(2)}(s,x)$, 
we have
\begin{itemize}
    \item 
    $ |J_{\eps}^{(1)}(s, x)| \le \eps  \|{\mathcal D}_\eps(s, x,\cdot)\|_{L^1_{\xi}}$,
    \item 
    $ |J_{\eps}^{(2)}(s, x)| \le C \varrho_{\eps}(s, x)^{1/2}  \|{\mathcal D}_\eps(s,x,\cdot)\|_{L^1_{\xi}}^{1/2}$,
    \item 
    $\| J_\eps^{(2)}(s,\cdot) \log_+^{1/2} |J_\eps^{(2)}(s,\cdot)| \|_{L^1_x}  \leq 
    C  \left[ \| \varrho_\eps(s,\cdot) \log_+ \varrho_\eps(s,\cdot) \|_{L^1_x} + \|{\mathcal D}_\eps(s,\cdot,\cdot)\|_{L^1_{x,\xi}} \right]$,
    \item $\| J_{\varepsilon}\log^{1/2}\log^{1/2}\max(J_{\varepsilon},e)\|_{L^1_{t,x}} \leq C(\|{\mathcal D}_\eps(s,\cdot,\cdot)\|_{L^1_{x,\xi}},\| \varrho_\eps(s,\cdot) \log_+ \varrho_\eps(s,\cdot) \|_{L^1_{t,x}} ) $ .
\end{itemize}
\end{proposition}
The first two estimates are similar to~\cite[Proposition 7.1]{MR2664460} for the Vlasov-Poisson-Fokker-Planck system. Here, we have additionally included the last two controls and we give a different proof.

\begin{proof}

The first two estimates are a direct consequence of Lemma~\ref{lem:pointwise_J}. The third estimate follows from the inequality, for $u\geq 1$, $v\geq 0$ and $uv\geq 1$,
\[
(uv)^{1/2} \log^{1/2} (uv) \leq u \log u + \sqrt 2 v.
\]

The last result is given for the sake of completeness and its technical proof is postponed to Appendix~\ref{ap:loglog}.
This concludes the proof of Proposition~\ref{lem:estimate_Jeps}.
\end{proof}

\noindent With these estimates, we can now prove the main result of this section.

\begin{proof}[Proof of Theorem \ref{thm:bounds_epsilonto0}] 
Estimate \ref{thm_eps:est1} follows by mass conservation. The next bounds are deduced from the energy equality 
\eqref{def:Henergy}-\eqref{helmholtz_dissipation} which we write as
\begin{equation}\label{eq:energy_sans_int_term}
\int_{\R^{2d}} \big[ 2D\log(f_\eps(t))+|\xi|^2 \big] f_\eps(t) \diff x \diff \xi + 2D\int_{0}^{t}\|{\mathcal D}_\eps(s,\cdot,\cdot)\|_{L^1_{x,\xi}} \diff s \leq \mathcal{F}(0),
\end{equation}
 where we ignore the nonnegative interaction term as it does not help in this computation. It is standard, see Appendix~\ref{appendix_a}, to conclude from this inequality that
\begin{equation}\label{eq:energy_pos_log}
\int_{\R^{2d}} \big[ 2D|\log(f_\eps(t))|+\frac{1}{2}|\xi|^2 \big] f_\eps(t) \diff x \diff \xi + D\, \|{\mathcal D}_\eps\|_{L^1_{t,x,\xi}} \leq \mathcal{F}(0) + C\left(\|\varrho_{\eps}\|_{L^{\infty}_t L^1_x}, \|x f^0\|_{L^{1}_{x,\xi}} \right).
\end{equation}

 The estimates \ref{thm_eps:est2} and \ref{thm_eps:est4} follow immediately. Then, estimate~\ref{thm_eps:est5} follows from estimate~\ref{thm_eps:est4} and the Csiszár-Kullback Inequality, see Lemma~\ref{lem:kullback_ineq}.
\\

\noindent Estimate \ref{thm_eps:est3} is also very standard and we reproduce the proof from \cite[Lemma 2.1]{MR2139941}. We consider the convex function $\psi(\varrho)=\varrho \log(\varrho)$ and  apply the Jensen inequality. We obtain 
\begin{multline*}
\varrho_\eps \log(\varrho_\eps)=\psi(\varrho_\eps) = \psi \left(\int_{\R^d} \f{f_\eps}{M}\ M\diff \xi \right)   \leq \int_{\R^{d}} \psi \left(\f {f_\eps}{M}\right) M \diff \xi =\\= \int_{\R^d} \f{f_\eps}{M} \bigg[\log f_\eps - \log M(\xi) \bigg] M \diff \xi  =\int_{\R^d} f_\eps \bigg[\log f_\eps +\f{|\xi|^2}{2D} \bigg] \diff \xi + C \int_{\R^d}f_\eps \diff \xi.
\end{multline*}
The conclusion follows by taking the absolute values of both sides and integrating with respect to~$x$.
\\

\noindent Finally, estimate \ref{thm_eps:est6} is a direct consequence of Proposition~\ref{lem:estimate_Jeps}, whereas \ref{thm_eps:est7} follows from \eqref{eq:bound_f_|x|}. Concerning the weak compactness of $\{\varrho_{\eps}\}$, it follows from estimates~\ref{thm_eps:est3} and~\ref{thm_eps:est7}. Then, the weak local compactness of $\{J_{\eps}\}$ is a direct consequence of Proposition~\ref{lem:estimate_Jeps} and the Dunford-Pettis theorem. Indeed, with the notations of Lemma~\ref{lem:pointwise_J}, $J_{\eps}^{1}$ converges strongly to 0 in $L^{1}_{t,x}$. For $J_{\eps}^{2}$ we first have the weak local compactness in $L^{1}_{t,x}$ thanks to the third estimate of Proposition~\ref{lem:estimate_Jeps}, bound~\ref{thm_eps:est3} and the Dunford Pettis theorem. To prove the global weak compactness we only need to prove it for $J_{\eps}^{(2)}$. We recall that, from Lemma~\ref{lem:pointwise_J}, we have
\[
 |J_{\eps}^{(2)}(s, x)| \le C \varrho_{\eps}(s, x)^{1/2}  \|{\mathcal D}_\eps(s,x,\cdot)\|_{L^1_{\xi}}^{1/2}.
\]
Therefore we can estimate with the Cauchy-Schwarz inequality
$$
\| J_{\eps}^{(2)}\, |x|^{1/2} \|_{L^1_{t,x}} \leq C\, \|\varrho_{\eps} \,|x| \|_{L^1_{t,x}}^{1/2} \, \|{\mathcal D}_\eps\|_{L^1_{t,x,\xi}}^{1/2}
$$
which yields global weak compactness in $L^{1}_{t,x}$ with the Dunford-Pettis theorem. This ends the proof.
\end{proof}

\section{The limit $\eps \to 0$}
\label{sec:limit_2}

We now perform the analysis allowing us to prove Theorem~\ref{thm:limit_eps}.  We take $\alpha=\eps$ where the parameter $\alpha$ defines the long range potential \eqref{as:alpha}. Note, however, that different scaling between $\alpha$ and $\eps$ could possibly be considered. \\

Recalling the mass balance equation~\eqref{eq:Mass} and the $\xi$-moment equation~\eqref{eq:Jeps},
our aim is to take the limit $\eps \to 0$ in these equations, and establish the relations
\beq
\p_t \varrho (t,x) + \dv \, J (t,x) = 0, 
\label{eq:MassLimit}
\eeq
\beq
J (t,x)= - D\nabla \varrho (t,x)  -  \varrho \nabla \Phi(t,x) , \qquad \Phi = -\delta\Delta[\omega^{S}\star\omega^{S}\star\varrho], 
\label{eq:JLimit}
\eeq
which are equivalent to \eqref{limit_equation}.
\\

A significant contribution comes from Theorem~\ref{thm:bounds_epsilonto0}. The entropy bound for~$\varrho_\eps$, see~\ref{thm_eps:est3}, and the $L^1$ bound on $J_\eps$, see Proposition~\ref{lem:estimate_Jeps}, we immediately conclude that
\\
$\bullet$ after extractions, $\varrho_{\eps}$ and $J_\eps (t,x)$ admit weak limits in $L^1_{t,x}$, $\varrho$ and  $J$, see also Theorem~\ref{thm:bounds_epsilonto0}, 
\\
$\bullet$ the equation~\eqref{eq:MassLimit} holds in the distributional sense. 
\\

The latter estimate on $J_\eps$ also tells us that $\eps^2 \p_t J_\eps (t,x)$ converges to $0$ in the distributional sense. Therefore, establishing the equation~\eqref{eq:JLimit} from equation~\eqref{eq:Jeps}, is reduced to proving the two local weak limits in $L^1_{t,x}$
\[
\int_{\R^d} \xi \otimes \xi \,  f_\eps (t,x, \xi) \diff \xi \to D \varrho (t,x)\, {\rm I}, \qquad 
\varrho_\eps \nabla \Phi_\eps \to \varrho \nabla \Phi(t,x).
\]
%
%
These follow directly from the following three lemmas
\begin{lemma}\label{lem:xi_tensor_xi_term} We have
\begin{equation*}
\int_{(0,T)\times \R^{d}}\left|\int_{\R^{d}}\xi\otimes\xi (f_{\eps}-\varrho_{\eps}M) \diff \xi\right|\diff x\diff t\xrightarrow[\eps\to 0]{} 0.
\end{equation*}
\end{lemma}
\begin{lemma} \label{prop:rho_strongly_L1}
The sequence $\{\varrho_\eps\}$ is precompact in $L^p_{t}L^{1}_{x}$ for every $1\le p<\infty$.
\end{lemma}
\begin{lemma} \label{lem:potential_boundedness}
The potential $\Phi_{\eps}(t,x)$ satisfies,  uniformly in $\eps \in (0,1)$, 
\begin{equation} \label{boundedness of potential}
  \|\Phi_{\eps}\|_\infty \le C,\quad \|\nabla\Phi_{\eps}\|_\infty\le C.
\end{equation}
Moreover, we have for every $1\le p<\infty$ the strong convergence in $L^{p}_{t}L^{\infty}_{x}$,
\begin{equation} \label{limit_phi}
    \Phi_\eps (t,x) \longrightarrow \Phi(t,x),\quad \nabla\Phi_\eps (t,x) \longrightarrow \nabla\Phi(t,x), \quad \Phi(t,x):=-\delta\Delta[\omega^{S}\star\omega^{S}\star\varrho(t,x)].
\end{equation}
\end{lemma}

The end of the proof of Theorem~\ref{thm:limit_eps} is thus to establish these results.

\begin{proof}[Proof of Lemma \ref{lem:potential_boundedness}]
Recalling the expressions of both long-range and short-range potentials and that $\alpha=\eps$, we see that 
\begin{equation*}
\Phi_\eps (t,x)=-\frac{1}{\eps^2} \int_{\R^d} \omega_\eps^L (z) \left[\omega^S \star \omega^S\star \varrho_\eps(t,x-z)-\omega^S \star \omega^S \star \varrho_\eps(t,x) \right] \diff z.
\end{equation*}
Let now set $y=\frac{z}{\eps}$, so that from \eqref{as:omega} we deduce that
\begin{equation*}
\Phi_\eps (t,x)=-\frac{1}{\eps^2} \int_{\R^d} \omega^L (y) \left[\omega^S \star \omega^S\star \varrho_\eps(t,x-\eps y)-\omega^S \star \omega^S \star \varrho_\eps(t,x) \right] \diff y.
\end{equation*}
Because the convolution terms are smooth (say $W^{3,\infty}$), we may use the Taylor expansion and obtain
\begin{equation*}
\begin{split}
\Phi_\eps(t,x)&=\frac{1}{\eps}\int_{\R^d}  \nabla_x [ \omega^S \star \omega^S \star \varrho_\eps(t,x)] \cdot y\  \omega^L (y)\diff y - \int_{\R^d} D^2_x [ \omega^S \star \omega^S \star \varrho_\eps(t,x)]y\cdot y \,\omega^L (y)\diff y +O(\eps)
\end{split}
\end{equation*}
where the term $O(\eps)$ converges to $0$ in $L^\infty$ since it is controlled by 
\[
C \eps \int_{\R^d} |y|^3 \omega^L (y)\| D^3_x\, \omega^S \star \omega^S \star \varrho_\eps(t,\cdot) \|_\infty \diff y,
\]
and we recall the uniform bound~\ref{thm_eps:est1}. Moreover, recalling \eqref{as:omega}, we see that the first term in the right-hand side vanishes and the Hessian matrix reduces to the Laplacian, so that
\begin{equation} \label{eq: boundedness_potential_3} 
\Phi_\eps(t,x)=-\delta \Delta_x \left[ \omega^S \star \omega^S \star \varrho_\eps(t,x) \right]+O(\eps)  
\end{equation}
from which we directly conclude from~\ref{thm_eps:est1}
\begin{equation*}
||\Phi_\eps||_\infty \leq C \quad \mbox{uniformly in $\eps \in (0,1)$}.
\end{equation*}
As far as $\nabla \Phi_\eps$ is concerned, the properties of convolution with respect to derivatives gives
\begin{equation*}
\nabla \Phi_\eps (t,x)=-\frac{1}{\eps^2} \int_{\R^d} \omega_\eps^L (z) \left[\nabla\omega^S \star \omega^S\star \varrho_\eps(t,x-z)-\nabla \omega^S \star \omega^S \star \varrho_\eps(t,x) \right] \diff z,
\end{equation*}
so that the $L^\infty_{t,x}$ bounded on $\nabla \Phi_\eps$ follows from the previous argument assuming now that $\omega^S \in W^{4,\infty}$. 
\\ 

It remains to show that $\Phi_\eps \rightarrow \Phi$ strongly in $L^{p}_{t}L^{\infty}_{x}$, the convergence of $\nabla\Phi_{\eps}$ uses the same arguments.
The convergence follows from~\eqref{eq: boundedness_potential_3} since we have 
\begin{equation*}
\Phi_\eps(t,x) - \Phi (t,x)= -\delta \left[\Delta \omega^S \star \omega^S \star (\varrho_\eps -\varrho)(t,x)\right] + O( \eps),
\end{equation*}
so that, thanks to the above control of the term $O(\eps)$ and properties of the convolution, 
\begin{equation} \label{eq: boundedness_potential_4}
\left\| \Phi_\eps - \Phi \right\|_{L^{p}_{t}L^{\infty}_{x}}\leq C \left\| \varrho_\eps -\varrho \right\|_{L^{p}_{t}L^{1}_{x}} + C \eps.
\end{equation}
Using Lemma~\ref{prop:rho_strongly_L1}, we obtain the result. 
\end{proof}

\begin{proof}[Proof of Lemma \ref{prop:rho_strongly_L1}]  
This result is a consequence of the compactness averaging lemma in kinetic theory \cite{GLPS88,PSaveraging}. Here, we use the following variant from  \cite[Lemma 4.2]{MR2299429}.

\begin{lemma}\label{lem:masmoudi}
Assume that $\{h^{\eps}\}$ is bounded in $L^{2}_{t,x,\xi}$, $\{h_{0}^{\eps}\}$ and $\{h_{1}^{\eps}\}$ are bounded in $L^{1}_{t,x,\xi}$. Moreover, suppose that
$$
\eps\partial_{t}h^{\eps}+\xi\cdot\nabla_{x}h^{\eps}=h_{0}^{\eps}+\nabla_{\xi}\cdot h_{1}^{\eps}. 
$$
\noindent Then, for all $\psi\in \mathcal{C}_{0}^{\infty}(\R^{d})$,
$$
\left\|\int_{\R^{d}}(h^{\eps}(t,x+y,\xi)-h^{\eps}(t,x,\xi))\, \psi(\xi)\diff\xi\right\|_{L^{1}_{t,x}}\to 0, 
$$
when $y\to 0$ uniformly in $\eps$.
\end{lemma}

To prove Lemma~\ref{prop:rho_strongly_L1}, we cannot apply this averaging lemma directly on $\{f_{\eps}\}$ because $\{f_{\eps}\}$ is not bounded in $L^{2}_{t,x,\xi}$ and we follow the argument in~\cite{MR2664460} which follows idea of renormalized solutions~\cite{DLBoltzmann89}. We fix $\nu > 0$ and we consider the functions 
$\beta_{\nu}(f)=\frac{f}{1+\nu f}$ with derivative $\beta_{\nu}'(f)=\frac{1}{(1+\nu f)^{2}}$.
Now we multiply \eqref{eq:V} by $\beta_{\nu}'(f)$ and obtain
$$
\eps\p_{t}\beta_{\nu}(f_{\eps})+\xi\cdot\nabla_{x}\beta_{\nu}(f_{\eps})=\f{(\varrho_{\eps}M-f)\beta_{\nu}'(f)}{\eps}-\nabla_{\xi}\cdot(F_{\eps}\beta_{\nu}(f_{\eps})).
$$
We verify assumptions of Lemma~\ref{lem:masmoudi}. From~\ref{thm_eps:est1} we see that $h^{\eps}=\beta_{\nu}(f_{\eps})$ is bounded in $L^{1}_{t,x,\xi} \cap L^{\infty}_{t,x,\xi}$ and hence in $L^{2}_{t,x,\xi}$ by interpolation. The $L^{1}_{t,x,\xi}$ bound on $h_{0}^{\eps}=\f{(\varrho_{\eps}M-f)\beta_{\nu}'(f_\eps)}{\eps}$ is deduced from~\ref{thm_eps:est5} and the $L^{\infty}_{t,x,\xi}$ bound on $\beta_{\nu}'(f_\eps)$. Finally, since $F_{\eps}$ is bounded in $L^{\infty}_{t,x}$ and $\beta_{\nu}(f_\eps)$ is bounded in $L^{1}_{t,x,\xi}$ we see that $h_{1}^{\eps}=-F_{\eps}\beta_{\nu}(f_\eps)$ is bounded in $L^{1}_{t,x,\xi}$.\\

\noindent The assumptions of Lemma \ref{lem:masmoudi} are satisfied and we obtain 

$$\left\|\int_{\R^{d}}(\beta_{\nu}(f_{\eps})(t,x+y,\xi)-\beta_{\nu}(f_{\eps})(t,x,\xi))\,\psi(\xi)\diff\xi\right\|_{L^{1}_{t,x}}\to 0,$$
when $y\to 0$, uniformly in $\eps$. As this is true for all $\nu>0$, Lemma \ref{lem:compact_beta} implies
\begin{equation}\label{eq:compactness_space_psi}
\left\|\int_{\R^{d}} (f_{\eps}(t,x+y,\xi)-f_{\eps}(t,x,\xi))\, \psi(\xi)\diff\xi\right\|_{L^{1}_{t,x}}\to 0,
\end{equation}
when $y\to 0$, uniformly in $\eps$.
\\

\noindent The final step is to remove the weight $\psi$ in the convergence \eqref{eq:compactness_space_psi} using uniform bound on $\{f_\eps \, |\xi|^2\}$. To this end, consider a sequence of functions $\{\psi_{n}(\xi)\}_{n}$ in $\mathcal{D}(\R^{d})$ such that $\psi_{n}(\xi)=1$ for $|\xi|\leq n$ and $\psi_n(\xi) = 0$ for $|\xi|\geq n+1$. Then,
$$
\left\|\int_{\R^{d}} (f_{\eps}(t,x,\xi) (1-\psi_n(\xi))\diff\xi\right\|_{L^{1}_{t,x}}   \leq \left\|\int_{|\xi|\geq n} f_{\eps}(t,x,\xi)\, \frac{|\xi|^2}{n^2}  \diff\xi\right\|_{L^{1}_{t,x}} \leq \frac{\| f_\eps |\xi|^2 \|_{L^1_{t,x,\xi}}}{n^2}
$$
and similarly for the term with $f_\eps(t, x+y, \xi)$.
Hence, we may choose first $n$ large enough and then for such $n$ apply \eqref{eq:compactness_space_psi} to deduce
\begin{equation}\label{eq:compactness_space}
\|\varrho_\eps(x+y) - \varrho_\eps(x)\|_{L^1_{t,x}} = \left\|\int_{\R^{d}} (f_{\eps}(t,x+y,\xi)-f_{\eps}(t,x,\xi)) \diff\xi\right\|_{L^{1}_{t,x}}\to 0,
\end{equation}
when $|y|\to 0$, uniformly in $\varepsilon>0$. This yields compactness in space.\\

\noindent From Lemma~\ref{lem:frechet-kolmogorov} we know that $\{\varrho_\eps\}$ is also compact in time, and as a result
\begin{equation*}
\begin{split}
&\int_{0}^{T-h} \hspace{-3pt} \int_{\R^d}  |\varrho_\eps(t+h,x+k) -\varrho_\eps(t,x)| \diff x \diff t \\
& \leq \int_{0}^{T-h} \hspace{-3pt}\int_{\R^d} |\varrho_\eps(t+h,x+k)-\varrho_\eps(t+h,x)| \diff t \diff x + \int_{0}^{T-h} \hspace{-3pt} \int_{\R^d} |\varrho_\eps(t+h,x)-\varrho_\eps(t,x)| \diff t \diff x \leq \theta(h,k),
\end{split}
\end{equation*}
where $\theta(h,k)\to 0$ whenever $|h|, |k| \to 0$ uniformly in $\eps$. This provides the equicontinuity of $\{\varrho_\eps\}$ in $L^1_{t,x}$ which provides us with local compactness in $x$. 

From \ref{thm_eps:est7} in Theorem~\ref{thm:bounds_epsilonto0} we know that
\begin{equation*}
 \sup_{0< \eps <1} \int_{(0,T)\times \R^d} |x \varrho_\eps(t,x)| \diff t \diff x<\infty,
\end{equation*}
and we obtain the strong convergence of the density in $L^{1}_{t,x}$ by Fréchet-Kolmogorov theorem, see also~\cite{MR916688}. Using Estimate~\ref{thm_eps:est1} we obtain by interpolation and~\cite[Theorem 1]{MR916688} the strong convergence in $L^{p}_{t}L^{1}_{x}$ for every $1\le p<\infty$
and this concludes the proof of Lemma~\ref{prop:rho_strongly_L1}. 
\end{proof}

\begin{proof}[Proof of Lemma \ref{lem:xi_tensor_xi_term}]
We adapt the proof of Lemma~\ref{lem:pointwise_J}. We write
\begin{equation*}
\begin{split}
R_{\eps}&:= \left|\int_{\R^{d}}\xi\otimes\xi (f_{\eps}-\varrho_{\eps}M) \diff \xi\right|\le 
\int_{\R^{d}} |\xi|^{2}|f_{\eps}-\varrho_{\eps}M|\diff \xi 
\\
&\le \int_{\left\{\left|\log(\f{f_{\eps}}{\varrho_{\eps} M})\right|\ge\f{|\xi|^{2}}{r}\right\}}{|\xi|^{2}} \varrho_{\eps} M\left|\f{f_{\eps}}{\varrho_{\eps} M}-1\right|\diff\xi 
+ \int_{\left\{\left|\log(\f{f_{\eps}}{\varrho_{\eps} M})\right|\le\f{|\xi|^{2}}{r}\right\}}{|\xi|^{2}}\varrho_{\eps} M\left|\f{f_{\eps}}{\varrho_{\eps} M}-1\right|\diff\xi =I_{1}+I_{2},
\end{split}
\end{equation*}
where $r$ is chosen later. For the first term, we just write 
\[
I_{1}\le r \int_{\R^d} \log\left(\f{f_{\eps}}{\varrho_{\eps} M}\right)\varrho_{\eps} M\left|\f{f_{\eps}}{\varrho_{\eps} M}-1\right|\diff\xi \le r \eps^{2} \,  \|{\mathcal D}_\eps\|_{L^1_{\xi}}.
\]
The term $I_2$ is decomposed in two parts: where $f_{\eps} \geq \varrho_\eps M$ and $f_{\eps} < \varrho_\eps M$. The resulting integrals are called $I_{2}^A$ and $I_2^B$. We only discuss $I_2^A$ as $I_2^B$ can be treated similarly as it was discussed in Lemma~\ref{lem:pointwise_J}. We use the Cauchy-Schwarz inequality to obtain
\begin{multline*}
I_{2}^A\le \left(\int_{\left\{0 \leq \log(\f{f_{\eps}}{\varrho_{\eps} M})\le\f{|\xi|^{2}}{r}\right\}}  {|\xi|^{4}} \f{\varrho_{\eps} M}{B(\xi)}\diff\xi \right)^{1/2} \cdot \\ 
 \cdot \left(\int_{\left\{0 \leq \log(\f{f_{\eps}}{\varrho_{\eps} M})\le\f{|\xi|^{2}}{r}\right\}}\varrho_{\eps} M\left|\f{f_{\eps}}{\varrho_{\eps} M}-1\right|^{2}B(\xi)\diff\xi \right)^{1/2} =:  I_{2}^{A,1}\cdot I_{2}^{A,2},
\end{multline*}
where, as before,  $B(\xi) = \frac{\log(A)}{A-1} =\f{|\xi|^{2}}{r(\exp(\f{|\xi|^{2}}{r})-1)}$,  with $A=A(\xi):=\exp(\f{|\xi|^{2}}{r})$. As in the proof of Lemma~\ref{lem:pointwise_J}, we have the inequality $\log(y)\ge (y-1)\f{\log(A)}{A-1}$ which yields with  $y=\f{f_{\eps}}{\varrho_{\eps}M}$
\begin{equation*}
I_{2}^{A,2}\le \left(\int_{\left\{0\leq \log(\f{f_{\eps}}{\varrho_{\eps} M})\le\frac{|\xi|^2}{r}\right\}}\varrho_{\eps} M\left|\f{f_{\eps}}{\varrho_{\eps} M}-1\right|\log\left(\f{f_{\eps}}{\varrho_{\eps}M}\right)\diff\xi \right)^{1/2} \leq \varepsilon\, \|D_{\eps}\|_{L^1_{\xi}}^{1/2}.    
\end{equation*}
Now we choose $r$ such that $M(\xi)\exp(\f{|\xi|^{2}}{r})=C\exp(-a|\xi|^{2})$ for some $a>0$. Then, we have
$$
\int_{\R^d} {|\xi|^{4}} \f{ M(\xi) }{B(\xi)} \diff \xi \leq  r\int_{\R^d} |\xi|^2\, M(\xi)\, \exp\left(\f{|\xi|^{2}}{r}\right)  \diff \xi 
\leq Cr \int_{\R^d} |\xi|^2\exp(-a|\xi|^{2}) \diff \xi =:C^2.
$$
It follows that $I_{2}^{A,1}\le C \varrho_{\eps}^{1/2} $.

Finally we get
\[
R_{\eps} \leq r \eps^{2} \,  \|{\mathcal D}_\eps\|_{L^1_{\xi}} + C\eps  \varrho_{\eps}^{1/2}\|D_{\eps}\|_{L^1_{\xi}}^{1/2}
\]
and, using the Cauchy-Schwarz inequality, the proof of Lemma \ref{lem:xi_tensor_xi_term} is concluded.

\end{proof}

This also concludes the proof of Theorem \ref{thm:limit_eps}.


\section{Conclusion} \label{conclusion}

We proved that macroscopic densities $\{\varrho_{\eps}\}$ formed from solutions of the Vlasov-Cahn-Hilliard equation \eqref{eq:V} converge to the solutions of non-local degenerate Cahn-Hilliard \eqref{limit_equation}. It is an open question whether one can obtain a local version of this equation by sending short-range interaction kernel $\omega^S$ to the Dirac mass $\delta_0$. One expects in the limit the local degenerate Cahn-Hilliard equation:
\begin{equation}\label{eq:local_deg_CH}
\p_t \varrho -D\Delta \varrho - \dv (\varrho \nabla \Phi)=0
\end{equation}
where $\Phi = -\delta\Delta\varrho$.
One can try to perform this limit either on equation \eqref{limit_equation} or directly on \eqref{eq:V}, by sending $\omega_{\alpha}^L \weaks \delta_0$, $\omega^S \weaks \delta_0$ together, see Figure \ref{fig:results}. Passing from \eqref{eq:V} to \eqref{eq:local_deg_CH}, the main difficulty is the lack of entropy which gives integrability of second-order derivatives in the nondegenerate Cahn-Hilliard. On the other hand, when one tries to pass to the limit from \eqref{limit_equation} to \eqref{eq:local_deg_CH}, the entropy is available but it yields estimates only on
$$
\Delta (\varrho \star \omega^S) \mbox{ in } L^2_t L^2_x, \qquad \nabla \sqrt{\varrho} \mbox{ in } L^2_t L^2_x.
$$
The minimal required information allowing to pass to the limit seems to be strong compactness of $\{\nabla \varrho\}$ in $L^2_t L^2_x$.\\

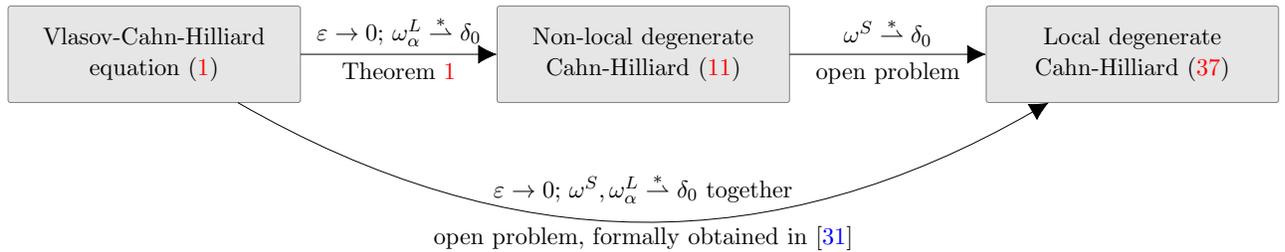
\begin{figure}[!htb]
\centering
\resizebox{\textwidth}{!}{%
\begin{tikzpicture}[node distance=7.8cm,auto]
\node [label] (init) {Vlasov-Cahn-Hilliard equation \eqref{eq:V}};
\node [label, right of=init] (node) {Non-local degenerate Cahn-Hilliard \eqref{limit_equation}};
\node [label, right of=node] (node1) {Local degenerate Cahn-Hilliard \eqref{eq:local_deg_CH}};

\draw[-{Latex[length=3mm, width=3mm]}] (init) -- node[below] {Theorem \ref{thm:limit_eps}} node[above] {$\eps\to 0$; $\omega_{\alpha}^L \weaks \delta_0$} (node);

\draw[-{Latex[length=3mm, width=3mm]}] (node) -- node[below] {open problem} node[above] {$\omega^S \weaks \delta_0$} (node1);

\draw[-{Latex[length=3mm, width=3mm]}] (init) to[bend right] node[midway,below,inner sep=2pt] {open problem, formally obtained in~\cite{takata2018simple}} node[midway,above,inner sep=8pt] {$\eps\to 0$; $\omega^S, \omega_{\alpha}^L \weaks \delta_0$ together} (node1);


\end{tikzpicture}
}

\caption{Relation between three types of the degenerate Cahn-Hilliard equations.}
\label{fig:results}
\end {figure}

\noindent Moreover, it is also open to prove whether we can add the "usual" double-well Cahn-Hilliard interaction potential in the system. In fact, as far as this modification is concerned, it is not even clear if there exists a solution to the Vlasov-Cahn-Hilliard equation when the potential $\Phi$ is a function of the density~$\varrho$. 

\section{Acknowledgments}
J.S. was supported by the National Agency of Academic Exchange project "Singular limits in parabolic equations" no. BPN/BEK/2021/1/00044. B.P. has received funding from the European Research Council (ERC) under the European Union's Horizon 2020 research and innovation programme (grant agreement No 740623).

\appendix

\section{Useful inequality and lower bound on the energy} \label{appendix_a}

We recall two lemmas which have been used in the proof of Theorem~\ref{thm:bounds_epsilonto0}. The first one is a variant of the Csiszar-Kullback inequality.

\begin{lemma}\label{lem:kullback_ineq}
Let $f, g \geq 0$ with $\|f\|_{1} = \|g\|_{1}$. Then,
$$
\|f-g\|_{1}^2 \leq  \|f\|_{1} \, \int_{\R^d} (f-g) \, (\log f - \log g)
$$
\end{lemma}
The second lemma is used to control $f \log_- (f)$ from $f \log f$, which immediately establishes the Inequality~\eqref{eq:bound_neg_part_of_log}.
\begin{lemma}\label{lem:bound_energy}
Let $\log_-(f):=\max\{-\log(f),0\}$. Then
\begin{equation}\label{eq:bound_neg_part_of_log}
\int_{\R^{2d}} 2D\log_-(f_\eps(t)) f_\eps(t) \diff x \diff \xi \leq  C\left(\norm{\varrho_{\eps}}_{L^{1}_{t,x}}, \|x f^0\|_{L^{1}_{x,\xi}} \right) + \int_{\R^{2d}} \frac{|\xi|^{2}}{4} f_\eps(t) \diff \xi \diff x  + D\int_{0}^{t}||{\mathcal D}_\eps(s,\cdot,\cdot)||_{L^1_{x,\xi}} \diff s.
\end{equation}

\end{lemma}

\begin{proof}[Proof of Lemma~\ref{lem:kullback_ineq}]
Let $\|f\|_{1} = \|g\|_{1} = 1$. Usual the Csiszar-Kullback inequality gives us
$$
\|f-g\|_{1}^2 \leq 2 \int_{\R^d} f \log \left( \frac{f}{g} \right).
$$
By symmetry of the (LHS) we have
$$
2 \|f-g\|_{1}^2 \leq 2 \int_{\R^d} f \log \left( \frac{f}{g} \right) + 2 \int_{\R^d} g \log \left( \frac{g}{f} \right) = 2 \int_{\R^d} (f-g) (\log f - \log (g)).
$$
The general case follows by rescaling.
\end{proof}

\begin{proof}[Proof of Lemma~\ref{lem:bound_energy}]

We proceed as in \cite[Proposition 5.1]{MR2664460}.

We divide the domain in two parts: 
$$
\Omega_1 := \left\{f_{\eps}>\exp\left({-\frac{|x|}{4}-\frac{|\xi|^{2}}{8D}}\right)\right\}, \qquad \qquad  
\Omega_2:= \left\{f_{\eps}\le \exp\left({-\frac{|x|}{4}-\frac{|\xi|^{2}}{8D}}\right)\right\},
$$
On $\Omega_1$, $\log_-(f_{\eps})$ is bounded so that we have
$$
f_{\eps}\log_-(f_{\eps})\le \left(\frac{|x|}{4}+\frac{|\xi|^{2}}{8D}\right)f_{\eps}, 
$$
while on $\Omega_2$, $f_{\eps} \leq 1$ so that $\sqrt{f_{\eps}}\log_-(f_{\eps})$ is bounded by some constant $C$. Hence, 
$$
f_{\eps}\log_-(f_{\eps})\le C\sqrt{f_{\eps}}\le C\, \exp\left({-\frac{|x|}{8}-\frac{|\xi|^{2}}{16D}}\right).
$$
It follows that
\begin{equation}\label{eq:log-ff_first_estimate}
\int_{\R^{2d}} \log_-(f_\eps(t)) f_\eps(t) \diff x \diff \xi \leq  \int_{\R^{2d}}C \exp\left({-\frac{|x|}{8}-\frac{|\xi|^{2}}{16D}}\right) +\left(\frac{|x|}{4}+\frac{|\xi|^{2}}{8D}\right) f_\eps(t) \diff \xi \diff x.
\end{equation}
Now, we only need to bound the term $\int_{\R^{2d}} \frac{|x|}{4} f_{\varepsilon}(t) \diff \xi \diff x$. For this, we first observe that
\[
\frac{\diff}{\diff t} \int_{\R^{d}} {|x|} f_{\varepsilon}(t) \diff \xi = \frac{1}{\varepsilon} \int_{\R^{d}} f_{\varepsilon}(t) \frac{x}{|x|} \xi \diff \xi = \frac{x}{|x|} J_{\varepsilon} \leq 2 \|{\mathcal D}_\eps(s,\cdot,\cdot)\|_{L^1_{\xi}}+ C \varrho_{\eps},
\]
where we have used Proposition~\ref{lem:estimate_Jeps} and Young's inequality (with  $\varepsilon\leq 1 $). Therefore, for all $t \geq 0$
\begin{equation}\label{eq:bound_f_|x|}
\int_{\R^{2d}} {|x|} f_{\varepsilon}(t) \diff \xi \diff x \leq \int_{\R^{2d}} {|x|} f^0 \diff \xi \diff x +C \norm{\varrho_{\eps}}_{L^{1}_{t,x}} + 2 \int_{0}^{t} \|{\mathcal D}_\eps(s,\cdot,\cdot)\|_{L^1_{x,\xi}} \diff s.
\end{equation}
Finally, equation \eqref{eq:log-ff_first_estimate} simplifies to give the desired result~\eqref{eq:bound_neg_part_of_log}.

\end{proof}

\section{Criteria for compactness}
\label{app:compactness}

\begin{lemma}[Compactness of $\beta_\nu(f_n)$ implies compactness of $f_n$]\label{lem:compact_beta}
Let $\{f_n(t,x,\xi)\}$ be a sequence such that $\{f_n\}$ and $\{f_n \log{f_n}\}$ are bounded in $L^1_{t,x,\xi}$. Let $\psi(\xi) \in C_c^{\infty}(\R^d)$. Suppose that for all $\nu>0$ and all $\varepsilon>0$, there exists $\delta(\nu, \varepsilon)$ such that,whenever $|y| \leq \delta(\nu, \varepsilon)$, 
$$
\left\| \int_{\R^d} (\beta_{\nu}(f_n(t,x+y,\xi)) - \beta_\nu(f_n(t,x,\xi)))\, \psi(\xi) \diff \xi \right\|_{L^1_{t,x}} \leq \varepsilon.
$$ 
Then, for all $\varepsilon>0$ there exists $\delta(\varepsilon)>0$ such that
$$
\left\| \int_{\R^d} (f_n(t,x+y,\xi) - f_n(t,x,\xi))\, \psi(\xi) \diff \xi \right\|_{L^1_{t,x}} \leq \varepsilon.
$$
\end{lemma}

\begin{proof}
First, we observe that
$$
|\beta_\nu(s) - s| \leq \left| \frac{s}{1+s\,\nu} - s \right| = \frac{\nu s^2}{1+\nu\,s} \leq \mbox{min}(\nu\,s^2, s).
$$
Therefore, for $M$ and $\nu$ to be chosen later
\begin{align*}
&\left\| \int_{\R^d} (f_n(t,x+y,\xi) - \beta_{\nu}(f_n(t,x+y,\xi)))\, \psi(\xi) \diff \xi \right\|_{L^1_{t,x}} \leq \\ 
& \qquad \quad \leq
\|\psi\|_{\infty}\, \nu \int_{f_n(t,x+y,\xi) \leq M} f_n^2(t,x+y,\xi) \diff \xi \diff x \diff t+ \|\psi\|_{\infty}\, \int_{f_n(t,x+y,\xi) \geq M} f_n(t,x+y,\xi) \diff \xi \diff x \diff t\\
& \qquad \quad \leq \|\psi\|_{\infty}\, \nu \, M \, \|f_n\|_{1} + \|\psi\|_{\infty} \, \frac{\|f_n \log f_n \|_{1}}{\log M}.
\end{align*}
Similarly,
$$
\left\| \int_{\R^d} (f_n(t,x,\xi) - \beta_{\nu}(f_n(t,x,\xi)))\, \psi(\xi) \diff \xi \right\|_{L^1_{t,x}} \leq 
\|\psi\|_{\infty}\, \nu \, M \, \|f_n\|_{1} + \|\psi\|_{\infty} \, \frac{\|f_n \log f_n \|_{1}}{\log M}.
$$
Let $\varepsilon > 0$. First, we choose $\nu$ and $M$ such that 
$$
\|\psi\|_{\infty}\, \nu \, M \, \|f_n\|_{1} + \|\psi\|_{\infty} \, \frac{\|f_n \log f_n \|_{1}}{\log M} \leq \frac{\varepsilon}{3}.
$$
Then, we take $\delta(\nu, \varepsilon/3)$ such that
$$
\left\| \int_{\R^d} (\beta_{\nu}(f_n(t,x+y,\xi)) - \beta_\nu(f_n(t,x,\xi)))\, \psi(\xi) \diff \xi \right\|_{L^1_{t,x}} \leq \varepsilon/3
$$
when $|y| \leq \delta(\nu, \varepsilon/3)$. The conclusion follows by the triangle inequality.
\end{proof}

\begin{lemma}
\label{lem:frechet-kolmogorov}
The sequence $\{\varrho_{\eps}\}$ from Lemma \ref{prop:rho_strongly_L1} is compact in time, i.e. 
\begin{equation*}
\lim_{|h|\to 0}\int_{0}^{T-h}\int_{\R^{d}}|\varrho_{\eps}(t+h,x)-\varrho_{\eps}(t,x)|\diff x \diff t = 0\quad \text{uniformly in $\eps$}.    
\end{equation*}
\end{lemma}

\noindent The proof of this lemma uses a sequence $(\varphi_{\delta})_{\delta>0}\in C_{c}^{\infty}(\R^{d})$ of standard mollifiers with mass~1 such that $\varphi_{\delta}(x)=\frac{1}{\delta^{d}}\varphi(\frac{x}{\delta})$ with $\varphi$ of mass 1 and compactly supported. Moreover
\begin{equation*}
 \|\nabla^{k}\varphi_{\delta}\|_{L^{1}(\R^{d})}\le\frac{C}{\delta^{k}},
\end{equation*}
and for any function $g\in L^{p}(\R^{d})$,
\begin{equation*}
\|g\star\varphi_{\delta}\|_{L^{p}(\R^{d})}\le \|\varphi_{\delta}\|_{L^{1}(\R^{d})}\|g\|_{L^{p}(\R^{d})}.	
\end{equation*}

\begin{proof}
We know that 
$$
\p_{t}\varrho_{\eps}+\nabla\cdot J_{\eps}=0
$$
where $J_{\eps}$ is bounded uniformly in $L^{1}_{t,x}$, see Proposition~\ref{lem:estimate_Jeps}. \\

\noindent Using the mollifiers with $\delta$ depending on $h$ to be specified later on, we first notice that
\begin{align*}
\int_{0}^{T-h}\int_{\R^{d}}|\varrho_{\eps}(t+h,x)-\varrho_{\eps}(t,x)|\diff x \diff t&\le \int_{0}^{T-h}\int_{\R^{d}}|\varrho_{\eps}(t,x)-\varrho_{\eps}(t,\cdot)\star\varphi_{\delta}(x)|\diff x \diff t\\
&+\int_{0}^{T-h}\int_{\R^{d}}|\varrho_{\eps}(t+h,x)-\varrho_{\eps}(t+h,\cdot)\star\varphi_{\delta}(x)|\diff x \diff t\\
&+\int_{0}^{T-h}\int_{\R^{d}}|\varrho_{\eps}(t+h,\cdot)\star\varphi_{\delta}(x)-\varrho_{\eps}(t,\cdot)\star\varphi_{\delta}(x)|\diff x \diff t.
\end{align*}
For the first and second terms, the computations are the same, hence, we only present it for the first term. Using the properties of the mollifiers and the compactness of $\varrho_{\eps}$ in space, we want to prove that
\begin{equation*}
\int_{0}^{T-h}\int_{\R^{d}}|\varrho_{\eps}(t,x)-\varrho_{\eps}(t,\cdot)\star\varphi_{\delta}(x)|\diff x \diff t\le  \theta(\delta).
\end{equation*}
where $\theta(\delta)\to 0$ when $\delta\to 0$ uniformly in $\eps$. We write 
\begin{equation*}
\int_{0}^{T-h}\int_{\R^{d}}|\varrho_{\eps}(t,x)-\varrho_{\eps}(t,\cdot)\star\varphi_{\delta}(x)|\diff x \diff t=\int_{0}^{T-h}\int_{\R^{d}}\left|\int_{\R^{d}}\varphi(y)(\varrho_{\eps}(t,x)-\varrho_{\eps}(t,x-\delta y))\diff y\right|\diff x \diff t.
\end{equation*}
Then we use Fubini's theorem and the fact that $\varphi$ is compactly supported in some compact set $K$ we obtain
\begin{equation*}
\int_{0}^{T-h}\int_{\R^{d}}\left|\int_{\R^{d}}\varphi(y)(\varrho_{\eps}(t,x)-\varrho_{\eps}(t,x-\delta y))\diff y\right|\diff x \diff t \le \int_{K}\|\tau_{\delta y}\varrho_{\eps}-\varrho_{\eps}\|_{L^{1}((0,T)\times\R^{d})}\diff y.
\end{equation*}
where $\tau_{x}$ is the translation operator in $x$ variable. Now we use the compactness in space obtained in~\eqref{eq:compactness_space}, so that
\begin{equation*}
\int_{K}\|\tau_{\delta y}\varrho_{\eps}-\varrho_{\eps}\|_{L^{1}((0,T)\times\R^{d})}\diff y\le |K|\sup_{y\in K}\|\tau_{\delta y}\varrho_{\eps}-\varrho_{\eps}\|_{L^{1}((0,T)\times\R^{d})}\le \theta(\delta).     
\end{equation*}
Therefore the first and the second term are bounded by $\theta(\delta)$ where $\theta(\delta)\to 0$ when $\delta \to 0$ uniformly in $\eps$. It remains to study the third term. The third term reads
\begin{align*}
\int_{0}^{T-h}\int_{\R^{d}}|\varrho_{\eps}(t+h,\cdot)\star\varphi_{\delta}(x)-\varrho_{\eps}(t,\cdot)\star\varphi_{\delta}(x)|\diff x \diff t&= \int_{0}^{T-h}\int_{\R^{d}}\left|\int_{t}^{t+h}\p_{t}\varrho_{\eps}(s,\cdot)\star\varphi_{\delta}(x)ds\right|\diff x \diff t\\
=\int_{0}^{T-h}\int_{\R^{d}}\left|-\sum_{i=1}^{d}\int_{t}^{t+h}J_{i}\star\p_{i}\varphi_{\delta}(s,x)ds\right|\diff x \diff t
&\le\sum_{i=1}^{d}\int_{0}^{T-h}\int_{\R^{d}}\int_{t}^{t+h}\left|J_{i}\star\p_{i}\varphi_{\delta}(s,x)\right|ds\diff x \diff t,
\end{align*}
where we used $J_{\eps}=(J_{i})_{i=1,...,d}$.  We perform the change of variables $v=\f{s-t}{h}$, use Fubini's theorem and obtain  

\begin{equation*}
 \int_{0}^{T-h}\int_{\R^{d}}\int_{t}^{t+h}\left|J_{i}\star\p_{i}\varphi_{\delta}(s,x)\right|ds\diff x \diff t=h    \int_{0}^{1}\int_{\R^{d}}\int_{0}^{T-h}\left|J_{i}\star\p_{i}\varphi_{\delta}(vh+t,x)\right|\diff t\diff x \diff v.
\end{equation*}
Then we use the change of variables $\tau=vh+t$ and obtain 
\begin{align*}
 h\int_{0}^{1}\int_{\R^{d}}\int_{0}^{T-h}\left|J_{i}\star\p_{i}\varphi_{\delta}(vh+t,x)\right|\diff t\diff x \diff v=   h\int_{0}^{1}\int_{\R^{d}}\int_{vh}^{T+h(v-1)}\left|J_{i}\star\p_{i}\varphi_{\delta}(\tau,x)\right|d\tau \diff x \diff v \le \f{h}{\delta}\|J_{\eps}\|_{L^{1}_{t,x}}.
\end{align*}
Using the $L^{1}_{t,x}$ bound on $J_{\eps}$ and taking $\delta=h^{1/2}$ we conclude. 

\end{proof}

\section{Uniqueness in $L^{\infty}$}
\label{app:uniqueness}

Let $d\geq3$. We are interested in the uniqueness of these solutions in the class of functions such that 
\begin{equation}\label{eq:class_for_uniqueness_slns}
\varrho \in L^{\infty}_{t,x} \cap L^{\infty}_t L^1_{x} \cap C^w_t L^1_x
\end{equation}
where $C^w_t L^1_x$ denotes the space of weakly continuous in time functions with values in $L^1_x$ . In this class, the definition of distributional solutions of Theorem~\ref{thm:limit_eps} can be formulated as follows: for every test function $\varphi\in C_{c}^{\infty}([0,T)\times\R^{d})$ we have, with 
\begin{equation*}
\begin{split}
-\int_{\R^{d}}\varrho^{0}\varphi(0,x)\diff x 
- \int_{0}^{T}\int_{\R^{d}}\varrho \p_{t}\varphi\diff x\diff t =
D\int_{0}^{T}\int_{\R^{d}}\varrho\Delta\varphi\diff x\diff t
-&\int_{0}^{T}\int_{\R^{d}}\varrho\nabla\Phi(\varrho)\cdot\nabla\varphi\diff x\diff t,   
\end{split}
\end{equation*}
where $\Phi(\varrho)=-\delta\Delta(\omega^{S}\star\omega^{S}\star\varrho)$ and $\varrho\in L^{\infty}_{t}L^{1}_{x}$.

By interpolation $\varrho$ belongs to every $L^{p}_{t,x}$, $1\le p\le\infty$ and so is $\nabla\Phi(\rho)$. Therefore this formulation implies

\begin{equation}\label{eq:weak_form_integrated_time_space}
\begin{split}
\int_{0}^{T}\langle \p_{t}\varrho,\varphi\rangle =
D\int_{0}^{T}\int_{\R^{d}}\varrho\Delta\varphi\diff x\diff t
-&\int_{0}^{T}\int_{\R^{d}}\varrho\nabla\Phi(\varrho)\cdot\nabla\varphi\diff x\diff t,   
\end{split}
\end{equation}

for every $\varphi \in L^1_t W^{1,1}_x \cap L^1_t \dot{H}_x^2$ where $\langle \cdot,\cdot\rangle$ denotes the dual pairing between $\dot{H}^{-2}$ and $\dot{H}^{2}$. 

Let $\varrho_{1},\varrho_{2}$ be two solutions as above with same initial data which satisfy $\varrho_{1},\varrho_{2}\in L^{\infty}_{t,x}$. The goal is to prove that $\varrho_{1}=\varrho_{2}$. We substract Equation~\eqref{eq:weak_form_integrated_time_space} for $\varrho_{2}$ and $\varrho_{1}$. Writing $\varrho=\varrho_{2}-\varrho_{1}$, we obtain 
\begin{equation}\label{eq:uniqueness}
\begin{split}
\int_{0}^{T}\langle \p_{t}\varrho,\varphi\rangle= D\int_{0}^{T}\int_{\R^{d}}\varrho\Delta\varphi\diff x \diff t-\int_{0}^{T}\int_{\R^{d}}&\varrho\nabla\Phi(\varrho_{2})\cdot\nabla\varphi\diff x \diff t \diff t-\int_{0}^{T}\int_{\R^{d}}\varrho_{1}\nabla\Phi(\varrho)\cdot\nabla\varphi\diff x \diff t. 
\end{split}
\end{equation}

We want to test \eqref{eq:uniqueness} with  $\varphi(t)=-\mathcal{N}\ast\varrho$ where $\mathcal{N}$ is the Newtonian potential so that $-\Delta \varphi = \varrho$. This is an admissible test function. Indeed, $\partial_{x_i, x_j} \varphi \in L^{\infty}_{t} L^2_{x}$ by the Calderon-Zygmund theory cf. \cite[Theorem 3.5, Chapter 3]{MR1616087}. Moreover, as $\nabla \mathcal{N} \in L^{\frac{d}{d-1},\infty}$ (i.e. weak $L^p$ spaces) we can use Young's convolutional inequality to deduce
$$
\|\nabla \varphi\|_{L^{\infty}_t L^2_x} \leq C\|\nabla \mathcal{N}\|_{L^{\frac{d}{d-1},\infty}} \|\varrho \|_{L^{\infty}_t L^{\frac{2d}{d+2}}_x} .
$$
Finally, $\varphi \in L^{\infty}_{t,x}$ cf. \cite[Lemma 1]{MR794002}. Therefore, testing \eqref{eq:uniqueness} with $\varphi$ we obtain 
\begin{equation*}
\f{1}{2} \int_{\R^{d}}|\nabla\varphi(T)|^{2}\diff x+D\int_0^T\int_{\R^{d}}\varrho^{2}=-\int_0^T\int_{\R^{d}}\varrho\nabla\Phi(\varrho_{2})\cdot\nabla\varphi -\int_0^T\int_{\R^{d}}\varrho_{1}\nabla\Phi(\varrho)\cdot\nabla\varphi. 
\end{equation*}

We denote by $I_{1}$ and $I_{2}$ the two terms of the right-hand side. Using $-\Delta \varphi = \varrho$ and the formula $\Delta\varphi\nabla\varphi=\nabla\cdot(\nabla\varphi\otimes\nabla\varphi)-\f{1}{2}\nabla|\nabla\varphi|^{2}$ we obtain
$$
I_1 = \int_0^T \int_{\R^d} \Delta \varphi \nabla \varphi \cdot \nabla \Phi(\varrho_2) \le C\int_0^T \int_{\R^d}  |D^{2}\Phi(\varrho_{2})||\nabla\varphi|^{2} \leq C \int_0^T \int_{\R^d} |\nabla\varphi|^{2}.
$$

as $|D^2 \Phi(\varrho_2)|$ can be bounded as in Lemma \ref{lem:potential_boundedness} only in terms of $\|\varrho_2\|_{L^{\infty}_{t,x}}$. For $I_{2}$ we recall that $\varrho_{1}$ is bounded in $L^{\infty}_{t,x}$. Using the Cauchy-Schwarz inequality it remains 
to see that $\norm{\nabla \Phi(\varrho)}_{L^{2}}\le C\norm{\nabla\varphi}_{L^{2}}$ which can be achieved by definition of $\Phi(\varrho)$ and $\varphi$ and the fact that convolutions commute with derivatives. Therefore 
\begin{equation*}
I_{2}\le C\norm{\nabla\varphi}_{L^{2}_{t,x}}^{2}.
\end{equation*}
Combining the previous results we obtain 
\begin{equation*}
\norm{\nabla\varphi(T,\cdot)}_{L^{2}}^{2} \le C \int_{0}^T \norm{\nabla\varphi}_{L^{2}}^{2},
\end{equation*}
so that $\norm{\nabla\varphi}_{L^{2}}^{2} = 0$ and the proof is concluded.

\section{Estimate on $J_{\varepsilon} \log^{1/2} \log^{1/2} \max(J_{\varepsilon},e)$}
\label{ap:loglog}
From Lemma~\ref{lem:pointwise_J} we recall that for $0<r\le 1$
\begin{equation*}
|J_{\eps}(s,x)|\le r \eps  \|{\mathcal D}_\eps(s, x,\cdot)\|_{L^1_{\xi}}+C\f{1}{r^{d}}\exp\left(\f{2C_{M}}{r^{2}}\right)\varrho_{\eps}(s,x)^{1/2}\|{\mathcal D}_\eps(s, x,\cdot)\|_{L^1_{\xi}}^{1/2}.
\end{equation*}
We can make further simplifications: applying a simple rescaling of $r$, ignoring $\varepsilon$, estimating $\frac{1}{r^d} \leq \exp(\frac{1}{r^d})$ and changing $r= \frac{1}{\alpha}$ we can assume
\begin{equation}\label{eq:estimate_J_with_alpha}
|J_{\eps}(s,x)|\le \frac{C}{\alpha} \|{\mathcal D}_\eps(s, x,\cdot)\|_{L^1_{\xi}}+C \exp\left(\alpha^2\right)\varrho_{\eps}(s,x)^{1/2}\|{\mathcal D}_\eps(s, x,\cdot)\|_{L^1_{\xi}}^{1/2}.
\end{equation}
To choose the best $\alpha$ in the inequality above, we let $u = \varrho_{\eps}$, $v = \|{\mathcal D}_\eps(s, x,\cdot)\|_{L^1_{\xi}}$ so that we can estimate
\begin{equation}\label{eq:Jeps_estimate_interms_ofuv}
|J_{\varepsilon}(s,x)| \leq C\, v \min_{1<\alpha < \infty} \left[\frac{1}{\alpha} + \exp\left({\alpha^2}\right) \sqrt{\frac{u}{v}}\, \right].
\end{equation}
\begin{lemma}\label{lem:minimum_in_estimate_Jeps}
Let $v \geq e$, $u \geq 0$, $v > e^2 \, u$. The minimum in \eqref{eq:Jeps_estimate_interms_ofuv} is attained for $\alpha > 1$ which is the unique solution of
$$
2 \,\alpha^3\, \exp(\alpha^2) = \sqrt{\frac{v}{u}}.
$$
For such $\alpha > 1$ we have
$$
v  \left[\frac{1}{\alpha} + \exp\left({\alpha^2}\right) \sqrt{\frac{u}{v}}\, \right] = v \left[\frac{1}{\alpha} + \frac{1}{2\alpha^3}\right] \leq \frac{2v}{\alpha}.
$$
 Then,
\begin{equation}\label{eq:cases_estimate_J_with_minimization}
\frac{2v}{\alpha} \leq \begin{cases}
\frac{2\sqrt{2}\,v}{\log^{1/2}_+\log^{1/2}_+ v} &\mbox{ if } v \geq u \log^{1/2}_+\ v,\\
2\,{u \log_+^{1/2} v} &\mbox{ if } v < u \log^{1/2}_+ v.
\end{cases}
\end{equation}
\end{lemma}
\begin{proof}
The first statement is a consequence of simple calculus and we only have to prove that the minimum is attained for $\alpha > 1$. This follows from
\begin{equation}\label{eq:estimate_alpha_big}
\sqrt{\frac{v}{u}} = 2\alpha^3 \exp(\alpha^2) \leq \exp(2 \alpha^2) \implies \frac{1}{2} \log \left( \frac{v}{u} \right) \leq \alpha^2.
\end{equation}
As $v > e^2\,u$, we deduce $\alpha > 1$.\\

We proceed to the estimates on $\frac{v}{2\alpha}$. Suppose that $v \geq u \log^{1/2}_+ v$. Then, we have
$$
\log v \geq \log u  + \log\log^{1/2}_+ v \implies \log^{1/2}_+\left(\frac{v}{u}\right) \geq {\log_+^{1/2}\log^{1/2}_+ v}
$$
(we use here $\frac{v}{u}>e^2$ and $v > e$ to write $\log_+$ instead of $\log$).
In view of \eqref{eq:estimate_alpha_big}, this gives lower bound on $\alpha$ which implies
$$
\frac{2v}{\alpha} \leq \frac{2\sqrt{2}v}{\log_+^{1/2}\log^{1/2}_+ v}.
$$
We are left with the case $v < u \log_+^{1/2} v$. In this case we estimate directly using $\alpha > 1$:
$$
\frac{2v}{\alpha} \leq 2v \leq 2\,{u \log_+^{1/2} v}.
$$
\end{proof}

We proceed to estimating $J_{\varepsilon} \log^{1/2} \log^{1/2} \max (J_{\varepsilon},e)$ in $L^1_{t,x}$. Let us observe that we can always restrict the set of integration to the points $(t,x)$ where $\|\mathcal{D}_{\varepsilon}\|_{L^1_{\xi}}$ is arbitrarily large. Indeed, given $M\geq e$, we estimate
\begin{align*}
\int_0^T\int_{\R^d}& J_{\varepsilon} \log^{1/2} \log^{1/2} \max(J_{\varepsilon},e) \leq \\ \leq &  \int_0^T\int_{\R^d} J_{\varepsilon} \log^{1/2} \log^{1/2} \max(J_{\varepsilon},e) \mathds{1}_{J_{\varepsilon} \leq M} + 
\int_0^T\int_{\R^d} J_{\varepsilon} \log^{1/2} \log^{1/2} \max(J_{\varepsilon},e) \, \mathds{1}_{\|\mathcal{D}_{\varepsilon}\|_{L^1_{\xi}} \leq e^2 \varrho_{\varepsilon}} \\
&+ 
\int_0^T\int_{\R^d} J_{\varepsilon} \log^{1/2} \log^{1/2} \max(J_{\varepsilon},e) \, \mathds{1}_{\|\mathcal{D}_{\varepsilon}\|_{L^1_{\xi}} > e^2 \varrho_{\varepsilon}} \, \mathds{1}_{ J_{\varepsilon} > M}.
\end{align*}
The first integral is bounded by $\|J_{\varepsilon}\|_{L^1_{t,x}} \, \log^{1/2} \log^{1/2} M$. For the second integral, we note that \eqref{eq:estimate_J_with_alpha} implies that $J_{\varepsilon} \leq C\, \varrho_{\varepsilon}$ so this integral is finite because we can use Young's
inequality and $\log x \leq x$
to get
$$
\varrho_{\varepsilon} \log^{1/2} \log^{1/2} \max( \varrho_{\varepsilon}, e) \leq 
\varrho_{\varepsilon} + \frac{1}{2} \varrho_{\varepsilon} \log\max( \varrho_{\varepsilon}, e).
$$
In the third integral, by estimate \eqref{eq:estimate_J_with_alpha} with $\alpha = 2$, we have $\|\mathcal{D}_{\varepsilon}\|_{L^1_{\xi}} \geq \frac{M}{C}$ for some constant $C$. It follows that $\|\mathcal{D}_{\varepsilon}\|_{L^1_{\xi}} $ can be assumed to be arbitrarily large by taking sufficiently large $M$. This allows us to apply Lemma~\ref{lem:minimum_in_estimate_Jeps}. \\

\noindent Splitting the domain of integration for two subsets as in Lemma \ref{lem:minimum_in_estimate_Jeps}, it is sufficient to prove that the following functions
\begin{align*}
P^1_{\varepsilon} &:= \frac{\|\mathcal{D}_{\varepsilon}\|_{L^1_{\xi}}}{\log_+^{1/2}\log_+^{1/2}\|\mathcal{D}_{\varepsilon}\|_{L^1_{\xi}}} \log_+^{1/2} \log_+^{1/2} \left(\frac{\|\mathcal{D}_{\varepsilon}\|_{L^1_{\xi}}}{\log_+^{1/2}\log_+^{1/2}\|\mathcal{D}_{\varepsilon}\|_{L^1_{\xi}}} \right), \\
P^2_{\varepsilon} &:= \varrho_{\varepsilon} \, \log_{+}^{1/2} \|\mathcal{D}_{\varepsilon}\|_{L^1_{\xi}}\,  \log_+^{1/2}\log_+^{1/2} \left(\varrho_{\varepsilon} \, \log_{+}^{1/2} \|\mathcal{D}_{\varepsilon}\|_{L^1_{\xi}}\right).
\end{align*}
are bounded in $L^1_{t,x}$ (here, we use that $\log_+^{1/2} \log_+^{1/2} v = \log^{1/2} \log^{1/2} \max(v,e)$). \\

For $P^1_{\varepsilon}$ (this is the limiting case!), we restrict to the values of $\|\mathcal{D}_{\varepsilon}\|_{L^1_{\xi}}$ so large that $\log_+^{1/2}\log_+^{1/2}\|\mathcal{D}_{\varepsilon}\|_{L^1_{\xi}} > 1$. Then,
$$
\log_+^{1/2} \log_+^{1/2} \left(\frac{\|\mathcal{D}_{\varepsilon}\|_{L^1_{\xi}}}{\log_+^{1/2}\log_+^{1/2}\|\mathcal{D}_{\varepsilon}\|_{L^1_{\xi}}} \right) \leq \log_+^{1/2} \log_+^{1/2} \left(\|\mathcal{D}_{\varepsilon}\|_{L^1_{\xi}} \right)
$$
so that $P^1_{\varepsilon} \leq \|\mathcal{D}_{\varepsilon}\|_{L^1_{\xi}}$.\\

For $P^2_{\varepsilon}$, we apply $\log x \leq x$, $\sqrt{x+y}\leq \sqrt{x}+\sqrt{y}$ and $2\,x\,y \leq {x^2 + y^2}$ to get
\begin{align*}
P^2_{\varepsilon} \leq \varrho_{\varepsilon} \, \log_{+}^{1/2} &\|\mathcal{D}_{\varepsilon}\|_{L^1_{\xi}}\,  \log_+^{1/2} \left(\varrho_{\varepsilon} \, \log_{+}^{1/2} \|\mathcal{D}_{\varepsilon}\|_{L^1_{\xi}}\right) \leq \\ &\leq
\varrho_{\varepsilon} \, \log_{+}^{1/2} \|\mathcal{D}_{\varepsilon}\|_{L^1_{\xi}}\,  \log_+^{1/2} \varrho_{\varepsilon} +
\varrho_{\varepsilon} \, \log_{+}^{1/2} \|\mathcal{D}_{\varepsilon}\|_{L^1_{\xi}}\,  \log_+^{1/2} \log_{+}^{1/2} \|\mathcal{D}_{\varepsilon}\|_{L^1_{\xi}}\\
&\leq \varrho_{\varepsilon} \log_+ \varrho_{\varepsilon}  + 
\varrho_{\varepsilon} \log_+\|\mathcal{D}_{\varepsilon}\|_{L^1_{\xi}} + \varrho_{\varepsilon} \log_+\|\mathcal{D}_{\varepsilon}\|_{L^1_{\xi}}
\end{align*}
so it is sufficient to prove that $\varrho_{\varepsilon} \, \log_{+} \|\mathcal{D}_{\varepsilon}\|_{L^1_{\xi}}$ is bounded in $L^1_{t,x}$. This follows from Fenchel-Young's inequality
$$
\varrho_{\varepsilon} \, \log_{+} \|\mathcal{D}_{\varepsilon}\|_{L^1_{\xi}} \leq
\varrho_{\varepsilon} \log \varrho_{\varepsilon} + \varrho_{\varepsilon}
+
\|\mathcal{D}_{\varepsilon}\|_{L^1_{\xi}}.
$$

%
%
%


\bibliographystyle{siam}  \bibliography{biblio}
\end{document}